\documentclass[a4paper,11pt]{article}

\usepackage{preamble-fv-simple}

\title{Classes of intersection digraphs with good algorithmic properties}

\author[1]{Lars Jaffke}

\author[2,3]{O-joung Kwon\thanks{Supported by the National Research Foundation of Korea (NRF) grant funded by the Ministry of Education (No. NRF-2018R1D1A1B07050294) and by Institute for Basic Science (IBS-R029-C1).}}

\author[1]{Jan Arne Telle}

\affil[1]{Department of Informatics, University of Bergen, Norway. \protect\\ \texttt{\{lars.jaffke, jan.arne.telle\}@uib.no}}
\affil[2]{Department of Mathematics, Incheon National University, South Korea.}
\affil[3]{Discrete Mathematics Group, Institute for Basic Science (IBS), Daejeon, South Korea. \protect\\ \texttt{ojoungkwon@gmail.com}}

\usepackage{macros}

\hypersetup{
    pdftitle=   {Classes of intersection digraphs with good algorithmic properties},
   pdfauthor=  {Lars Jaffke, O-joung Kwon, and Jan Arne Telle}
}

\date{May 4, 2021}

\begin{document}

\maketitle

\begin{abstract}
An intersection digraph is a digraph where every vertex $v$ is represented by an ordered pair $(S_v, T_v)$ of sets such that there is an edge from $v$ to $w$ if and only if $S_v$ and $T_w$ intersect. An intersection digraph is reflexive if $S_v\cap T_v\neq \emptyset$ for every vertex $v$. Compared to well-known undirected intersection graphs like interval graphs and permutation graphs, not many algorithmic applications on intersection digraphs have been developed. 

Motivated by the successful story on algorithmic applications of intersection graphs using a graph width parameter called mim-width, we introduce its directed analogue called `bi-mim-width' and prove that various classes of reflexive intersection digraphs have bounded bi-mim-width. In particular, we show that as a natural extension of $H$-graphs, reflexive $H$-digraphs have linear bi-mim-width at most $12\abs{E(H)}$, which extends a bound on the linear mim-width of $H$-graphs [On the Tractability of Optimization Problems on $H$-Graphs. Algorithmica 2020].

For applications, we introduce a novel framework of directed versions of locally checkable problems, that streamlines the definitions and the study of many problems in the literature and facilitates their common algorithmic treatment. We obtain unified polynomial-time algorithms for these problems on digraphs of bounded bi-mim-width, when a branch decomposition is given. Locally checkable problems include \textsc{Kernel}, \textsc{Dominating Set}, and \textsc{Directed $H$-Homomorphism}.
\end{abstract}

\section{Introduction}

An undirected graph $G$ is an \emph{intersection graph} if there exists a family $\{S_v:v\in V(G)\}$ of sets such that two vertices $v$ and $w$ in $G$ are adjacent if and only if $S_v$ and $S_w$ intersect. A famous example is an interval graph which is an intersection graph of intervals on a line. There are various intersection graph classes of algorithmic interest; for example, permutation graphs, chordal graphs, circular-arc graphs, circle graphs and so on. We refer to \cite{BrandstadtLS1999} for an overview. It is well known that many NP-hard optimization problems can be solved in polynomial time on simple intersection graph classes.

Intersection graphs have been generalized to digraphs. Beineke and Zamfirescu~\cite{BeinekeZ1982} first introduced intersection digraphs under the name of connection digraphs. A digraph $G$ is an \emph{intersection digraph} 
if there exists a family $\{(S_v, T_v):v\in V(G)\}$ of ordered pairs of sets, called a \emph{representation}, such that there is an edge from $v$ to $w$ in $G$ if and only if $S_v$ intersects $T_w$. Note that we add a loop on a vertex $v$ if $S_v$ and $T_v$ intersect.
Sen, Das, Roy, and West~\cite{SenDRW1989} considered \emph{interval digraphs} that are intersection digraphs $G$ represented by a family $\{(S_v, T_v):v\in V(G)\}$ where all the sets in $\{S_v, T_v:v\in V(G)\}$ are intervals on a line, and provided characterizations of interval digraphs, which are analogous to characterizations for interval graphs. Later on, as direct analogues of circular-arc graphs and permutation graphs, circular-arc digraphs~\cite{SenDW1989} and permutation digraphs~\cite{Muller1997} have been considered. 

M\"uller~\cite{Muller1997} obtained a polynomial-time recognition algorithm for interval digraphs.
However, surprisingly, we could not find any literature that studied algorithmic applications on natural digraph problems on general interval digraphs. We observe that interval digraphs contain, for each integer $n$, some orientation of the $(n\times n)$-grid (see Proposition~\ref{prop:intunbounded}). As interval graphs do not contain an induced subgraph isomorphic to the $1$-subdivision of the claw, underlying undirected graphs of interval digraphs are very different from interval graphs.
Perhaps, this makes it difficult to find algorithmic applications of interval digraphs. 

For this reason, some restrictions of interval digraphs have been considered.
Prisner~\cite{Erich1994} introduced \emph{interval nest digraphs}~\cite{Erich1994} that are interval digraphs where $T_v\subseteq S_v$ for each vertex $v$, and proved that \textsc{Independent Dominating Set} and \textsc{Kernel} can be solved in polynomial time on interval nest digraphs, if a representation is given. Feder, Hell, Huang, and Rafiey~\cite{FederHHR2012} introduced another type of an interval digraph, called an \emph{adjusted interval digraph}, which has an interval digraph representation where for each vertex $v$, $S_v$ and $T_v$ have the same left endpoint. This class has been studied in connection with the \textsc{List Homomorphism} problem. A common point of these variants is that the corresponding representation requires that for each vertex $v$, $S_v$ and $T_v$ intersect, and so $v$ has a loop. We say that a digraph is \emph{reflexive} if every vertex has a loop.

In this paper, we obtain unified polynomial-time algorithms for several digraph problems on many classes of reflexive intersection digraphs. To do so, we introduce a new digraph width parameter, called \emph{bi-mim-width}, which is a directed analogue of the mim-width of an undirected graph introduced by Vatshelle~\cite{VatshelleThesis}. Belmonte and Vatshelle~\cite{BelmonteV2013} showed that many intersection graph classes have bounded mim-width, including interval graphs and permutation graphs. 
Briefly speaking, the bi-mim-width of a digraph $G$ is defined as a branch-width, with a cut function that measures, for a vertex partition $(A, B)$ of $G$, the sum of the sizes of maximum induced matchings in two bipartite digraphs, one  induced by edges from $A$ to $B$, and the other  induced by edges from $B$ to $A$. This is similar to how rank-width is generalized to bi-rank-width for digraphs~\cite{Kante2007, KanteR2013}. We formally define bi-mim-width and linear bi-mim-width in Section~\ref{sec:bimimwidth}. We compare bi-mim-width and other known width parameters. The mim-width of an undirected graph is exactly the half of the bi-mim-width of the digraph obtained by replacing each edge with bi-directed edges, and this observation can be used  to argue that a bound on the bi-mim-width of a class of digraphs implies a bound on the mim-width of a certain class of undirected graphs.

Telle and Proskurowski~\cite{TelleP1997} introduced locally checkable vertex subset problems (LCVS problems) and vertex partition problems (LCVP problems). LCVS problems include \textsc{Independent Set} and \textsc{Dominating Set} and LCVP problems include  \textsc{$H$-Homomorphism}. Bui-Xuan, Telle, and Vatshelle~\cite{Bui-XuanTV13} showed that all these problems can be solved in time XP parameterized by mim-width, if a corresponding decomposition is given.

\begin{table}[tb]
    \centering
    \begin{tabular}{|l|l|l|l|l|l|l|l|}   \hline
    $\sigma^+$ & $\sigma^-$ & $\rho^+$ & $\rho^-$ & Standard name \\
    \hline \hline
       $\{0\}$ & $\{0\}$ & $\mathbb{N}\setminus \{0\}$ & $\mathbb{N}$ & Kernel  ~\cite{NeumannM1944} \\
        $\{0,...,k-1\}$ & $\{0\}$ & $\{i: i \geq l\}$ & $\mathbb{N}$ & $(k,l)$-out Kernel \cite{Ramoul}\\
     $\mathbb{N}$ & $\mathbb{N}$ & $\mathbb{N}$ & $\mathbb{N}\setminus \{0\}$ &  Dominating set  \cite{GLP} \\
      $\{0\}$ & $\{0\}$ & $\mathbb{N}$ & $\mathbb{N}\setminus \{0\}$ & Independent Dominating set  \cite{CMC} \\
      $\mathbb{N}$ & $\mathbb{N}$ &  $\mathbb{N}\setminus \{0\}$ & $\mathbb{N}$ & in-Dominating set  \cite{Fu68}\\
         $\mathbb{N}$ & $\mathbb{N}$ & $\mathbb{N}\setminus \{0\}$ & $\mathbb{N}\setminus \{0\}$ &  Twin Dominating set \cite{CDSS}\\
       $\mathbb{N}$ & $\mathbb{N}$ & $\mathbb{N}$ & $\{i:i\ge k\}$ & $k$-Dominating set \cite{OBB} \\
      $\mathbb{N}$ & $\mathbb{N}\setminus \{0\}$ & $\mathbb{N}$ &    $\mathbb{N}\setminus \{0\}$ &  Total Dominating set \cite{Arumugan}\\
       $\{0\}$ & $\{0\}$ & $\mathbb{N}$ & $ \{1\}$ & Efficient (Closed) Dominating set \cite{Bange}\\
      $\mathbb{N}$ & $\{1\}$ & $\mathbb{N}$ & $ \{1\}$ & Efficient Total Dominating set \cite{Schaudt}\\
     $\{k\}$ & $\{k\}$ & $\mathbb{N}$ & $\mathbb{N}$ & $k$-Regular Induced Subdigraph \cite{CML}\\
   \hline
    \end{tabular}
    \caption{Examples of $(\sigma^+, \sigma^-, \rho^+, \rho^-)$-sets. For any row there is an associated NP-complete problem, usually maximizing or minimizing the cardinality of a set with the property. Some properties are known under different names; e.g. Efficient Total Dominating sets are also called Efficient Open Dominating sets, and here even the existence of such a set in a digraph $G$ is NP-complete, as it corresponds to deciding if $V(G)$ can be partitioned by the open out-neighborhoods of some $S \subseteq V(G)$. If rows A and B have their in-restrictions and out-restrictions swapped for both $\sigma$ and $\rho$ (i.e. $\sigma^+$ of row A equals $\sigma^-$ of row B and vice-versa, and same for $\rho^+$ and $\rho^-$), then a row-A set in $G$ is always a row-B set in the digraph with all arcs of $G$ reversed; this is the case for Dominating set vs in-Dominating set and for Kernel vs Independent Dominating set.}
    \label{tab1}
    \end{table}
    
We introduce directed LCVS and LCVP problems. 
A directed LCVS problem is represented as a $(\sigma^+, \sigma^-, \rho^+, \rho^-)$-problem for some $\sigma^+, \sigma^-, \rho^+, \rho^- \subseteq \bN$, and  it asks to find a maximum or minimum vertex set $S$ in a digraph $G$ such that for every vertex $v$ in $S$, the numbers of out/in-neighbors in $S$ are contained in $\sigma^+$ and $\sigma^-$, respectively, and for every vertex $v$ in $V(G)\setminus S$, the numbers of out/in-neighbors in $S$ are contained in $\rho^+$ and $\rho^-$, respectively. See Table~\ref{tab1} for several examples that appear in the literature. In particular, it includes the \textsc{Kernel} problem, which was introduced by von Neumann and Morgenstern~\cite{NeumannM1944}.

A directed LCVP problem is represented by a $(q \times q)$-matrix $D$ for some positive integer $q$, where for all $i, j\in \{1, \ldots, q\}$, $D[i,j]=(\mu_{i,j}^+,\mu_{i,j}^-)$ for some $\mu_{i,j}^+,\mu_{i,j}^-\subseteq \bN$. The problem asks to find a vertex partition of a given digraph into $X_1, X_2, \ldots, X_q$ such that for all $i, j\in [q]$, the numbers of out/in-neighbors of a vertex of $X_i$ in $X_j$ are contained in $\mu_{i,j}^+$ and $\mu_{i,j}^-$, respectively.
\textsc{Directed $H$-Homomorphism} is a directed LCVP problem:
For a digraph $H$ on vertices $\{1, \ldots, q\}$, 
we can view a homomorphism from a digraph $G$ to $H$ as a $q$-partition $(X_1, \ldots, X_q)$ of $V(G)$ such that we can only have an edge from $X_i$ to $X_j$ if the edge $(i, j)$ is present in $H$. See \cref{tab:lcvp}.
The \textsc{Oriented $k$-Coloring} problem, introduced by Sopena~\cite{Sopena1997}, asks whether there is a homomorphism to some orientation of a complete graph on at most $k$ vertices, and can therefore be reduced to a series of directed LCVP problems.
Several works in the literature concern problems of $2$-partitioning the vertex sets of digraphs into parts with degree constraints either inside or between the parts of the partition~\cite{AlonEtAl2020,BangJensenEtAl2018,BangJensenEtAl2019,BangJensenChristiansen2018,BangJensenEtAl2016}. 
All of these problems can be observed to be LCVP problems as well, 
see~\cref{tab:lcvp}.
Note that in the LCVP-framework, we can consider $q$-partitions for any fixed $q \ge 2$, for all problems apart from \textsc{$2$-Out-Coloring}. 
This fails for \textsc{$q$-Out-Coloring}, since this problem asks for a $q$-coloring with no monochromatic out-neighborhood.
\begin{table}[tb]
	\centering
	\begin{tabular}{|l|c|c|}
		\hline
		Problem name & $q$ & LCVP $(q \times q)$-matrix $D$ \\
		\hline\hline
		Directed $H$-Homomorphism~\cite{HellNesetril2004} & $\card{V(H)}$ & 
		 $\forall (i, j) \in E(H)\colon D[i, j] = (\bN, \bN)$ \\
		 &	& $\forall (i, j) \notin E(H)\colon D[i, j] = (\{0\}, \{0\})$ \\
		\hline
		Oriented $k$-Coloring~\cite{Courcelle1994,Sopena2016} & $k$ &
			$\bigvee_{H \colon \overrightarrow{K_k}} \mbox{Directed $H$-Homomorphism}$ \\
				\hline
		$\exists$ $(\sigma^+, \sigma^-,\rho^+,\rho^-)$-set  [This paper]&
		$2$ & 
		$\begin{pmatrix}
		    (\sigma^+, \sigma^-) & (\bN, \bN) \\
		    (\rho^+, \rho^-) &  (\bN, \bN)
		\end{pmatrix}$ \\
		\hline
		$(\delta^+ \ge k_1, \delta^- \ge k_2)$-Partition~\cite{BangJensenEtAl2016} & $2$ & 
		$\begin{pmatrix}
		    (\{j \colon j \ge k_1\}, \bN) & (\bN, \bN) \\
		    (\bN, \bN) & (\bN, \{j \colon j \ge k_2\})
		\end{pmatrix}$
		\\
		\hline
		$(\delta^+ \ge k_1, \delta^+ \ge k_2)$-Partition~\cite{BangJensenChristiansen2018} &
		$2$ & 
		$\begin{pmatrix}
		    (\{j \colon j \ge k_1\}, \bN) & (\bN, \bN) \\
		    (\bN, \bN) & (\{j \colon j \ge k_2\}, \bN)
		\end{pmatrix}$ \\
		\hline
		$(\Delta^+ \le k_1, \Delta^+ \le k_2)$-Partition~\cite{BangJensenEtAl2018} 
		    & $2$ &
		$\begin{pmatrix}
		    (\{j \colon j \le k_1\}, \bN) & (\bN, \bN) \\
		    (\bN, \bN) & (\{j \colon j \le k_2\}, \bN)
		\end{pmatrix}$ \\
		\hline
		 $(\delta^+ \ge k_1, \delta^- \ge k_2)$-Bipartite-Partition~\cite{BangJensenEtAl2019}
		 & $2$ &
		 $\begin{pmatrix}
		    (\bN, \bN) & (\{j\colon j \ge k_1\}, \bN) \\
		    (\bN, \{j \colon j \ge k_2\}) & (\bN, \bN)
		 \end{pmatrix}$
		 \\
		\hline
		 $(\delta^+ \ge k_1, \delta^+ \ge k_2)$-Bipartite-Partition~\cite{BangJensenEtAl2019}
		 & 2 &
		 $\begin{pmatrix}
		    (\bN, \bN) & (\{j \colon j \ge k_1\}, \bN) \\
		    (\{j \colon j \ge k_2\}, \bN) & (\bN, \bN)
		 \end{pmatrix}$
		 \\
		\hline
		$2$-Out-Coloring~\cite{AlonEtAl2020} 
		& $2$ & 
		$\begin{pmatrix}
		    (\bN \setminus \{0\}, \bN) & (\bN \setminus \{0\}, \bN) \\
		    (\bN \setminus \{0\}, \bN) & (\bN \setminus \{0\}, \bN)
        \end{pmatrix}$
		\\
		\hline
	\end{tabular}
	\caption{Examples of directed LCVP problems. 
    For every row there are choices of values for which the problems are NP-complete.
	For Directed $H$-Homomorphism let 
	$V(H) = \{1, \ldots, \card{V(H)}\}$ and denote by $H \colon \protect\overrightarrow{K_k}$ that $H$ is a tournament on $k$ vertices.
}
\label{tab:lcvp}
\end{table}

\begin{mainthm}\label{thm:xpalgorithm}
Directed LCVS and LCVP problems can be solved in time XP parameterized by bi-mim-width, when a branch decomposition is given.
\end{mainthm}

Furthermore, we show that the distance variants of directed LCVS problems, for instance \textsc{\mbox{Distance-$r$} Dominating Set} can be solved in polynomial time on digraphs of bounded bi-mim-width. Another natural variant is the \textsc{$k$-Kernel} problem (see~\cite[Section 8.6.2]{digraphbook2018}), which asks for a kernel in the $(k-1)$-th power of a given digraph.
To show this, we prove that  the $r$-th power of a digraph of bi-mim-width $w$ has bi-mim-width at most $rw$ (Lemma~\ref{lem:powerbimim}). For undirected graphs, there is a bound that does not depend on $r$~\cite{JaffkeKST2019}, but we were not able to obtain such a bound for the directed case.  
\begin{mainthm}\label{thm:xpalgorithmdistance}
Distance variants of directed LCVS problems can be solved in time XP parameterized by bi-mim-width, when a branch decomposition is given.
\end{mainthm}

We provide various classes of digraphs of bounded bi-mim-width. We first summarize our results in the following theorem and give the background below. We illustrate the bounds in \cref{fig:classes}.

\begin{mainthm}\label{thm:bddbimim}
\begin{enumerate}
    \item\label{refint} Given a reflexive interval digraph, one can output a linear branch decomposition of bi-mim-width at most $2$ in polynomial time. On the other hand, interval digraphs have unbounded bi-mim-width.
    \item Given a representation of an adjusted permutation digraph $G$, one can construct in polynomial time a linear branch decomposition of $G$ of bi-mim-width at most $4$. Permutation digraphs have unbounded bi-mim-width.
    \item Given a representation of an adjusted rooted directed path digraph $G$, one can construct in polynomial time a branch decomposition of $G$ of bi-mim-width at most $2$. Rooted directed path digraphs have unbounded bi-mim-width and adjusted rooted directed path digraphs have unbounded linear bi-mim-width. 
    \item Let $H$ be an undirected graph. Given a representation of a reflexive $H$-digraph $G$, one can construct in polynomial time a linear branch decomposition of $G$ of bi-mim-width at most $12\abs{E(H)}$. $P_2$-digraphs, which are interval digraphs, have unbounded bi-mim-width.
    \item Let $H$ be an undirected graph. Given a nice $H$-convex digraph $G$ with its bipartition $(A,B)$, one can construct in polynomial time a linear branch decomposition of $G$ of bi-mim-width at most $12\abs{E(H)}$. $P_2$-convex digraphs have unbounded bi-mim-width.
    \item Tournaments and directed acyclic graphs have unbounded bi-mim-width.
\end{enumerate}
\end{mainthm}

\begin{figure}
    \centering
    \includegraphics[width=\textwidth]{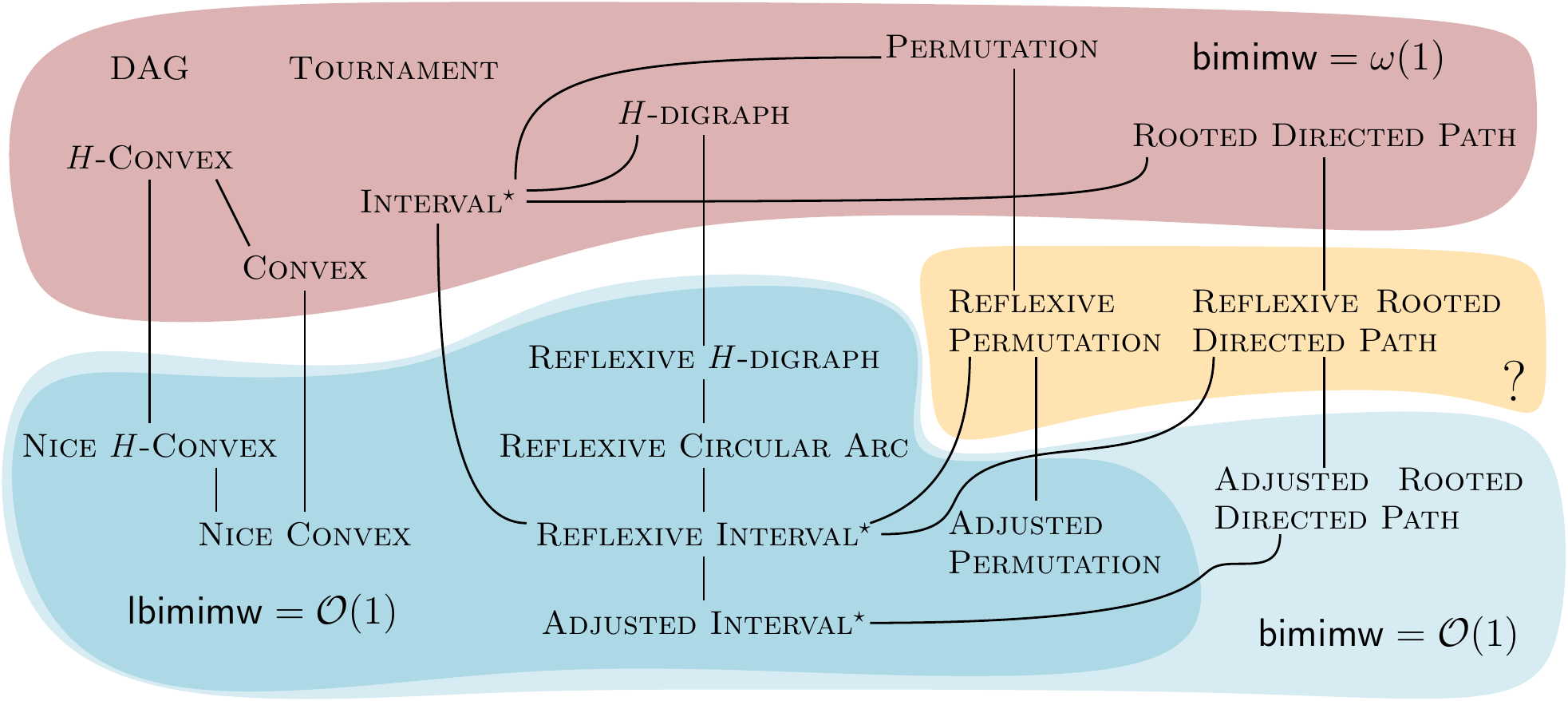}
    \caption{Digraph classes with bounds on their (linear) bi-mim-width. For graph classes marked with $^\star$ there are polynomial-time algorithms to compute representations of their members. 
    If digraph class $A$ is depicted above $B$ and there is an edge between $A$ and $B$ then $B \subseteq A$.}
    \label{fig:classes}
\end{figure}

((i). Interval digraphs) Recall that M\"uller~\cite{Muller1997} devised a recognition algorithm for interval digraphs, which also outputs a representation. By testing the reflexivity of a digraph, we can recognize reflexive interval digraphs, and output its representation. We convert it into a linear branch decomposition of bi-mim-width at most 2. On the other hand, interval digraphs generally have unbounded bi-mim-width.
By Theorem~\ref{thm:xpalgorithm}, we can solve all directed LCVS and LCVP problems on reflexive interval digraphs in polynomial time. This extends the polynomial-time algorithms for \textsc{Independent Dominating Set} and \textsc{Kernel} on interval nest digraphs given by Prisner~\cite{Erich1994}. 

\medskip

((ii). Permutation digraphs) A \emph{permutation digraph} is an intersection digraph of pairs of line segments whose endpoints lie on two parallel lines. M\"{u}ller~\cite{Muller1997} considered permutation digraphs under the name `matching diagram digraph', and observed that every interval digraph is a permutation digraph. Therefore, permutation digraphs have unbounded bi-mim-width.
We say that a permutation digraph is \emph{adjusted} if there exists one of the parallel lines, say $\Lambda$, such that for all $v\in V(G)$, $S_v$ and $T_v$ have the same endpoint in $\Lambda$. We show that every adjusted permutation digraph has linear mim-width at most 4. 

\medskip

((iii). Rooted directed path digraphs)  It is known that chordal graphs have unbounded mim-width~\cite{KangKST2017, Mengel18}. As restrictions of chordal graphs, it has been shown that rooted directed path graphs, and more generally, leaf power graphs have mim-width at most $1$~\cite{JaffkeKST2019}, while they have unbounded linear mim-width.  
A \emph{rooted directed path digraph} is an intersection digraph of pairs of directed paths in a rooted directed tree (every node is reachable from the root), and it is \emph{adjusted} if for every vertex $v$, the endpoint of $S_v$ that is farther from the root is the same as the endpoint of $T_v$ that is farther from the root. We show that every adjusted rooted directed path digraph has bi-mim-width at most $2$. Since this class includes the biorientations of trees, it has unbounded linear bi-mim-width.

\medskip

((iv). $H$-digraphs)
For an undirected graph $H$, an \emph{$H$-graph} is an undirected intersection graph of connected subgraphs in an $H$-subdivision, introduced by B\'{\i}r\'{o}, Hujter, and Tuza~\cite{BiroHT1992}. For example, interval graphs and circular-arc graphs are $P_2$-graphs and $C_3$-graphs, respectively. 
Fomin, Golovach, and Raymond~\cite{FominGR2020} showed that $H$-graphs have linear mim-width at most $2\abs{E(H)}+1$.
Motivated by $H$-graphs, we introduce an \emph{$H$-digraph} that is the intersection digraph of pairs of connected subgraphs in an $H$-subdivision (where $H$ and its subdivision are undirected). We prove that reflexive $H$-digraphs have linear bi-mim-width at most $12\abs{E(H)}$. This extends the linear bound of Fomin et al.~\cite{FominGR2020} for $H$-graphs.

\medskip

((v). $H$-convex digraphs)
  For an undirected graph $H$, a bipartite digraph $G$ with bipartition $(A, B)$ is an \emph{$H$-convex digraph}, if there exists a subdivision $F$ of $H$ with $V(F)=A$ such that for every vertex $b$ of $B$, each of the set of out-neighbors and the set of in-neighbors of $v$ induces a connected subgraph in $F$. We say that an $H$-convex digraph is \emph{nice} if for every vertex $b$ of $B$, there is a bi-directed edge between $b$ and some vertex of $A$.  Note that $H$-convex graphs, introduced by  Brettell, Munaro, and Paulusma~\cite{BrettellMP2020}, can be seen as nice $H$-convex digraphs, by replacing every edge with bi-directed edges. We prove that nice $H$-convex digraphs have linear bi-mim-width at most $12\abs{E(H)}$.
 This implies that $H$-convex graphs have linear mim-width at most $6\abs{E(H)}$. 
  Brettell et al.~\cite{BrettellMP2020} showed that for every tree $T$ with maximum degree $\Delta$ and $t$ branching nodes, $T$-convex graphs have mim-width at most $\max\{2\lfloor (\Delta/2)^2 \rfloor, 2\Delta-1\} +2^{t-1}\Delta$. As such trees have at most $(t\Delta)/2$ edges, our result implies an upper bound of $3t\Delta$ which is an improvement for $t\ge 5$.


 \medskip
 ((vi). Directed acyclic graphs and tournaments) 
 We show that if $H$ is the underlying undirected graph of a digraph $G$, then the bi-mim-width of $G$ is at least the mim-width of $H$. Using this, we can show that acyclic orientations of grids have unbounded bi-mim-width.
 We also prove that tournaments have unbounded bi-mim-width. This refines an argument that they have unbounded bi-rank-width~\cite[Lemma 9.9.11]{digraphbook2018} and shows that a class of digraphs may have unbounded bi-mim-width even though their underlying undirected graphs have bounded mim-width (even bounded rank-width).

We can summarize our algorithmic results as follows.

\begin{maincor}
Given a reflexive interval digraph, or a representation of either an adjusted permutation digraph,  or an adjusted rooted  directed path digraph,  or a reflexive $H$-digraph, or a nice $H$-convex digraph, we can solve all directed LCVS and LCVP problems, and distance variants of directed LCVS problems, in polynomial time. 
\end{maincor}

The paper is organized as follows. In Section~\ref{sec:prelim}, we introduce basic notations. In Section~\ref{sec:bimimwidth}, we formally introduce bi-mim-width and compare with other known width parameters. In Section~\ref{sec:classes}, we prove Theorem~\ref{thm:bddbimim}, and in Section~\ref{sec:application}, we prove Theorems~\ref{thm:xpalgorithm} and \ref{thm:xpalgorithmdistance}.

\section{Preliminaries}\label{sec:prelim}

For a positive integer $n$, we use the shorthand 
$[n] \defeq \{1, \ldots, n\}$.

\paragraph*{Undirected Graphs.}
We use standard notions of graph theory and refer to~\cite{Diestel2010} for an overview.
All undirected graphs considered in this work are finite and simple.
For an undirected graph $G$, we denote by $V(G)$ the vertex set of $G$ and $E(G)$ the edge set of $G$.
For an edge $\{u, v\} \in E(G)$, we may use the shorthand `$uv$'. 

For an undirected graph $G$ and two disjoint vertex sets $A, B \subseteq V(G)$, we denote by 
$G[A, B]$ the bipartite graph on bipartition $(A, B)$ such that $E(G[A, B])$ is exactly the set of edges of $G$ incident with both $A$ and $B$.


\paragraph*{Digraphs.}
All digraphs considered in this work are finite and have no multiple edges, but may have loops.
For a digraph $G$, we denote by $V(G)$ its vertex set and by $E(G) \subseteq V(G) \times V(G)$ its edge set. 
We say that an edge $(u, v) \in E(G)$ is directed from $u$ to $v$. 

For a digraph $G$, the undirected graph obtained by replacing every edge with undirected edge and then removing multiple edges is called its \emph{underlying undirected graph}. For an undirected graph $G$, a digraph obtained by replacing every edge $uv$ with one of $(u,v)$ and $(v,u)$ is called its \emph{orientation}, and the digraph obtained by replacing every edge $uv$ with two directed edges $(u,v)$ and $(v,u)$ is called its \emph{biorientation}.
 
For a digraph $G$ and two disjoint vertex sets $A, B \subseteq V(G)$, we denote by 
$G[A \rightarrow B]$ the bipartite digraph on bipartition $(A, B)$
with edge set $E(G[A, B]) = E(G) \cap (A \times B)$, 
and denote by $G[A, B]$ the bipartite digraph on bipartition $(A,B)$ with edge set $E(G[A \to B]) \cup E(G[B \to A])$. We denote by $M_G[A\to B]$ the matrix whose columns are indexed by~$A$ and rows are indexed by $B$ such that for $a\in A$ and $b\in B$, 
    $M_G[A\to B](a,b)=1$ if there is an edge from $a$ to $b$ and $0$ otherwise. 

A \emph{tournament} is an orientation of a complete graph.

\paragraph*{Common notations.}
Let $G$ be an undirected graph or a digraph.
A set $M$ of edges in $G$ is a \emph{matching} if no two edges share an endpoint, and it is an \emph{induced matching} if there are no edges in $G$ meeting two distinct edges in $M$.
We denote by $\nu(G)$ the maximum size of an induced matching  of $G$.

For two undirected graphs or two directed graphs $G$ and $H$, we denote by $G\cap H=(V(G)\cap V(H), E(G)\cap E(H))$ and $G\cup H=(V(G)\cup V(H), E(G)\cup E(H))$.

For a vertex set $A$ of $G$, we denote by $\bar{A}:=V(G)\setminus A$.
A vertex bipartition $(A, \bar{A})$ of $G$ for some vertex set $A$ of $G$ will be called a \emph{cut}. A cut $(A, \bar{A})$ of $G$ is \emph{balanced} if $\abs{V(G)}/3<\abs{A}\le 2\abs{V(G)}/3$.

For two vertices $u, v \in V(G)$, the \emph{distance} between $u$ and $v$, denoted by $\dist_G(u, v)$ or simply $\dist(u, v)$, is the length of the shortest path from $u$ to $v$ (if $G$ is a digraph, then we consider directed paths).
For a positive integer $d$, we denote by $G^d$ the graph obtained from $G$ by, for every pair $(x,y)$ of vertices in $G$, adding an edge from $x$ to $y$ if there is a path of length at most $d$ from $x$ to $y$ in $G$. We call it the \emph{$d$-th power} of $G$.

\section{Bi-mim-width}\label{sec:bimimwidth}

In this section, we introduce the bi-mim-width of a digraph. 
For an undirected graph $G$ and $A\subseteq V(G)$, let $\mimval_G(A):=\nu(G[A, \bar{A}])$.
For a digraph $G$ and $A\subseteq V(G)$,
\begin{itemize}
    \item let $\mimval^+_G(A):=\nu(G[A\to \bar{A}])$ and $\mimval^-_G(A):=\nu(G[\bar{A}\to A])$ and
    \item  let $\bimimval_G(A):=\mimval^+_G(A)+\mimval^-_G(A)$.
\end{itemize}

A tree is \emph{subcubic} if it has at least two vertices and every internal vertex has degree $3$.
A tree $T$ is a \emph{caterpillar} if it contains a path $P$ such that every vertex in $V(T)\setminus V(P)$ has a neighbor in $P$.
Let $G$ be an undirected graph or a digraph.
A \emph{branch decomposition} of $G$ is a pair $(\decT, \decf)$ of a subcubic tree $\decT$ and a bijection $\decf$ from $V(G)$ to the leaves of $\decT$.
If $\decT$ is a caterpillar, then $(\decT, \decf)$ is called a \emph{linear branch decomposition} of $G$.
\begin{definition}[Bi-mim-width]
    Let $G$ be a digraph and let $(\decT, \decf)$ be a branch decomposition of $G$. 
    For each edge $e \in E(\decT)$, let $\decT_A$ and $\decT_B$ be the connected components of $\decT - e$.
    Let $(A_e, B_e)$ be the cut of $G$ where
    $A_e$ is the set of vertices that $\decf$ maps to the leaves in $\decT_A$ and
    $B_e$ is the set of vertices that $\decf$ maps to the leaves in $\decT_B$.
    The \emph{bi-mim-width} of $(\decT, \decf)$ is
    $\bimimw(\decT, \decf) \defeq \max_{e \in E(\decT)}
        \left(\bimimval_G(A_e)\right).$
    The \emph{bi-mim-width} of $G$, denoted by $\bimimw(G)$, is the minimum bi-mim-width of any branch decomposition of $G$.
    The \emph{linear bi-mim-width} of $G$, denoted by $\linbimimw(G)$, is the minimum bi-mim-width of any linear branch decomposition of $G$.
\end{definition}

This is motivated by the mim-width of an undirected graph introduced by Vatshelle~\cite{VatshelleThesis}. 

\begin{definition}[Mim-width]
    Let $G$ be an undirected graph and let $(\decT, \decf)$ be a branch decomposition of $G$. 
    For each edge $e \in E(\decT)$, let $\decT_A$ and $\decT_B$ be the two connected components of $\decT - e$.
    Let $(A_e, B_e)$ be the cut of $G$ where
    $A_e$ is the set of vertices that $\decf$ maps to the leaves in $\decT_A$ and
    $B_e$ is the set of vertices that $\decf$ maps to the leaves in $\decT_B$.
    The \emph{mim-width} of $(\decT, \decf)$ is
    $ \mimw(\decT, \decf) \defeq \max_{e \in E(\decT)}
        \mimval_G(A_e).$
    The \emph{mim-width} of $G$, denoted by $\mimw(G)$, is the minimum mim-width of any branch decomposition of $G$.
    The \emph{linear mim-width} of $G$, denoted by $\linmimw(G)$, is the minimum mim-width of any linear branch decomposition of~$G$.
\end{definition}

The following two lemmas are clear by definition.
\begin{lemma}\label{lem:subdigraph}
Let $G$ be a digraph and let $H$ be an induced subdigraph of $G$. Then $\bimimw(H)\le \bimimw(G)$ and $\linbimimw(H)\le \linbimimw(G)$. 
\end{lemma}

\begin{lemma}\label{lem:biorientation}
Let $G$ be an undirected graph and let $H$ be the biorientation of $G$. Then for every vertex partition $(A,B)$ of $G$, we have $\nu (G[A, B])= \frac{ \nu (H[A \to B])+\nu(H[B\to A])}{2}$. In particular, we have $\mimw(G)=\frac{\bimimw(H)}{2}$.
\end{lemma}

We show that if a digraph $G$ has small bi-mim-width, then its underlying undirected graph has small mim-width. But the other direction does not hold; the class of tournaments has unbounded bi-mim-width. 

\begin{lemma}\label{lem:underlying}
Let $G$ be a digraph and let $H$ be the underlying undirected graph of $G$. Then $\mimw(H)\le \bimimw(G)$ and $\linmimw(H)\le \linbimimw(G)$. On the other hand, the class of tournaments has unbounded bi-mim-width, while their underlying undirected graphs have linear mim-width~$1$.
\end{lemma}
\begin{proof}
    Let $(\decT, \decf)$ be a branch-decomposition of $H$, and assume that it has mim-width $t$. 
    Then, $\decT$ has an edge $e$ inducing a cut $(A_e, B_e)$ such that $\mimval_H(A_e)= t$.
    Let $M$ be a maximum induced matching of $H[A_e, B_e]$, and 
    let $(M_1, M_2)$ be the partition of $M$ such that original edges in $M_1$ are contained in $G[A_e\to B_e]$ and original edges of $M_2$ are contained in 
    $G[B_e\to A_e]$.
    It shows that the sum of the sizes of the maximum induced matchings in $G[A_e\to B_e]$ and in $G[B_e\to A_e]$ is at least $t$.
    This implies that $(\decT, \decf)$ has bi-mim-width at least $t$, as a branch-decomposition of $G$. 
    As we chose $(\decT, \decf)$ arbitrarily, we conclude that $G$ has bi-mim-width at least $\mimw(H)$. The same argument obtained by replacing a branch-decomposition with a linear branch-decomposition shows that $\linmimw(H)\le \linbimimw(G)$.
    
    We prove the second statement. For every integer $n\ge 2$, we define  $G_n$ as the graph on the vertex set $\{v_{i,j}:i,j\in [n]\}$ satisfying that 
    \begin{itemize}
        \item for all $i\in [n]$ and $j_1, j_2\in [n]$ with $j_1<j_2$, there is an edge from $v_{i, j_1}$ to $v_{i, j_2}$,
        \item for all $i\in [n-1]$ and $j\in [n]$, there is an edge from $v_{i, j}$ to $v_{i+1, j}$, 
        \item for all $i_1, i_2, j_1, j_2\in [n]$ with $i_1<i_2$, 
        \begin{itemize}
            \item if $i_2-i_1\equiv 0\pmod 3$ and $j_1\ge j_2$, then there is an edge from $v_{i_1, j_1}$ to $v_{i_2, j_2}$, 
            \item if $i_2-i_1\equiv 0\pmod 3$ and $j_1< j_2$, then  there is an edge from $v_{i_2, j_2}$ to $v_{i_1, j_1}$,
            \item if $i_2-i_1\equiv 1\pmod 3$ and $(i_2, j_2)\neq (i_1+1, j_1)$, then there is an edge from $v_{i_2, j_2}$ to $v_{i_1, j_1}$,
            \item if $i_2-i_1\equiv 2\pmod 3$, then there is an edge from $v_{i_1, j_1}$ to $v_{i_2, j_2}$.
        \end{itemize}
    \end{itemize}
    For each $i\in [n]$, we let $R_i:=\{v_{x,i}:x\in [n]\}$ and 
    $C_i:=\{v_{i,y}:y\in [n]\}$.
    
    We claim that for every positive integer $k$, $G_{18k}$ has bi-mim-width at least $k$. 
    Suppose for contradiction that there is a branch decomposition of $G$ of bi-mim-width at most $k-1$.
    Therefore, there is a balanced cut $(A, B)$ of $G$ where 
    $\bimimval_G(A)\le k-1$.
    Let $n:=18k$.
    
    We divide into two cases.    \medskip
    
    (Case 1. For every $i\in [n]$, $R_i$  contains a vertex of $A$ and a vertex of $B$.)
    Then there exists $a_i\in [n-1]$ for each $i\in [n]$ such that $v_{a_i, i}$ and $v_{a_i+1, i}$ are contained in distinct sets of $A$ and $B$. Let $I_1\subseteq [n]$ be a set of size at least $n/2$ such that either 
    \begin{itemize}
        \item for all $i\in I_1$, $v_{a_i,i}\in A$, or 
        \item for all $i\in I_1$, $v_{a_i,i}\in B$.    \end{itemize}
    Without loss of generality, we assume that for all $i\in I_1$, $v_{a_i,i}\in A$. The proof will be symmetric when $v_{a_i,i}\in B$. Furthermore, we take a subset $I_2$ of $I_1$ of size at least $n/6$ such that 
    integers in $\{a_i:i\in I_2\}$ are pairwise congruent modulo 3.

    Now, we verify that $\{(v_{a_i,i}, v_{a_i+1, i}):i\in I_2\}$ is an induced matching in $G[A\to B]$.
    Let $x,y\in I_2$ be distinct integers and assume that $a_x\le a_y$. If $a_x=a_y$, then $(a_x+1)-a_y\equiv 1 \pmod 3$, and thus there is an edge from $v_{a_x+1, x}$ to $v_{a_y, y}$ and similarly, there is an edge from $v_{a_y+1, y}$ to $v_{a_x, x}$.
    Assume that $a_x<a_y$. Then $a_y-(a_x+1)\equiv 2 \pmod 3$ and $(a_y+1)-a_x\equiv 1 \pmod 3$, and thus there is an edge from $v_{a_x+1,x}$ to $v_{a_y, y}$ and there is an edge from $v_{a_y+1, y}$ to $v_{a_x, x}$.
    This shows that there are no edges between $\{v_{a_x, x}, v_{a_x+1, x}\}$ and $\{v_{a_y, y}, v_{a_y+1, y}\}$ in $G[A\to B]$, and therefore  $\{(v_{a_i,i}, v_{a_i+1, i}):i\in I_2\}$ is an induced matching in $G[A\to B]$ of size at least $n/6=3k$.
    This contradicts the assumption that $\bimimval_G(A)\le k-1$.
    
    \medskip
    (Case 2. For some $j\in [n]$, $R_j$ is fully contained in one of $A$ and $B$.) Without loss of generality, we assume that $R_j$ is contained in $A$. Since $\abs{B}>\abs{V(G)}/3$, there is a subset $I_1\subseteq [n]$ such that $\abs{I_1}\ge n/3$, and for each $i\in I_1$, $a_{i, b_i}\in B$ for some $b_i\in [n]$.
    We take a subset $I_2$ of $I_1$ of size at least $n/9$ where all integers in $I_2$ are pairwise congruent modulo 3.
    Lastly, we take a subset $I_3$ of $I_2$ of size at least $n/18$ such that either
    \begin{itemize}
        \item for all $i\in I_3$, $b_i>j$, or
        \item for all $i\in I_3$, $b_i<j$.
    \end{itemize}
    
    First, we assume that $b_i>j$
    for all $i\in I_3$.
    We verify that $\{(v_{i,j}, v_{i, b_i}):i\in I_3\}$ is an induced matching in $G[A\to B]$. Let $x, y\in I_3$ with $x<y$. As $x$ and $y$ are congruent modulo 3, there is an edge from $v_{x, b_{x}}$ to $v_{y, j}$ and there is an edge from $v_{y, b_{y}}$ to $v_{x, j}$. It shows that there are no edges between $\{v_{x, j}, v_{x, b_x}\}$ and $\{v_{y, j}, v_{y, b_y}\}$ in $G[A\to B]$, and therefore  $\{(v_{x,j}, v_{x, b_x}):x\in I_3\}$ is an induced matching in $G[A\to B]$ of size at least $n/18=k$.
    It contradicts the assumption that $\bimimval(A)\le k-1$.
    The argument when $b_i<j$ for all $i\in I_3$ is similar.
 \end{proof} 


We argue that directed tree-width~\cite{JohnsonRST2001} and bi-mim-width are incomparable. 

\begin{lemma}\label{lem:directedtw}
Directed tree-width and bi-mim-width are incomparable.
\end{lemma}
\begin{proof}
First, the class of all acyclic orientations of undirected grids has directed tree-width~$1$ but has unbounded bi-mim-width, by Lemma~\ref{lem:underlying}. Second, the class of all possible digraphs obtained from a directed cycle by replacing a vertex with a set of vertices with same in-neighbors and out-neighbors that are not adjacent to each other, has bi-mim-width at most~$4$, but has unbounded directed tree-width.  To see that such a digraph has large directed tree-width, one can create a subdivision of a cylindrical grid as a subgraph which has large directed tree-width~\cite{JohnsonRST2001}.
\end{proof}

We compare the bi-mim-width with the bi-rank-width of a digraph, introduced by Kant\'e~\cite{Kante2007}. Kant\'e and Rao~\cite{KanteR2013} later generalized this notion to edge-colored graphs.    For a digraph $G$ and $A\subseteq V(G)$,
\begin{itemize}
    \item let $\cutrk^+_G(A):=\rank(M_G[A\to \bar{A}])$ and $\cutrk^-_G(A):=\rank(M_G[\bar{A}\to A])$ and
    \item  let $\bicutrk_G(A):=\cutrk^+_G(A)+\cutrk^-_G(A)$,
\end{itemize}
where the rank of a matrix is computed over the binary field.

\begin{definition}[Bi-rank-width]
  Let $G$ be a digraph and let $(\decT, \decf)$ be a branch decomposition of $G$. 
    For each edge $e \in E(\decT)$, let $\decT_A$ and $\decT_B$ be the two connected components of $\decT - e$.
    Let $(A_e, B_e)$ be the cut of $G$ where
    $A_e$ is the set of vertices that $\decf$ maps to the leaves in $\decT_A$ and
    $B_e$ is the set of vertices that $\decf$ maps to the leaves in $\decT_B$. 
    The \emph{bi-rank-width} of $(\decT, \decf)$ is
    $\birw(\decT, \decf) \defeq \max_{e \in E(\decT)}
        \bicutrk_G(A_e)$.
    The \emph{bi-rank-width} of $G$, denoted by $\birw(G)$, is the minimum bi-rank-width of any branch decomposition of $G$.
    The \emph{linear bi-rank-width} of $G$, denoted by $\linbirw(G)$, is the minimum bi-rank-width of any linear branch decomposition of $G$.
\end{definition}


We can verify that for every digraph $G$, $\bimimw(G)\le \birw(G)$. 
Interestingly, we can further show that for every positive integer $r$, the bi-mim-width of the $r$-th power of $G$ is at most the bi-rank-width of $G$. This does not depend on the value of $r$.


\begin{lemma}\label{lem:powerbirw}
Let $r$ and $w$ be positive integers. If $(\decT, \decf)$ is a branch-decomposition of a digraph $G$ of bi-rank-width $w$, then it is a branch-decomposition of $G^r$ of bi-mim-width at most $w$.
\end{lemma}
\begin{proof}
It is sufficient to prove that for every ordered vertex partition $(A, B)$ of $G$, we have $\nu(G^r[A\to B])\le \rank (M_G[A\to B])$. Assume $\rank (M_G[A\to B])=t$ and suppose for contradiction that $\nu(G^r[A\to B])\ge t+1$.

Let $\{(a_i, b_i):i\in [t+1]\}$ be an induced matching of $G^r[A\to B]$ with $\{a_i:i\in [t+1]\}\subseteq A$. It means that for each  $i\in [t+1]$, there is a directed path $P_i$ of length at most $r$ from $a_i$ to $b_i$ in $G$ such that the paths in $\{P_i:i\in [t+1]\}$ are pairwise vertex-disjoint. We choose an edge $(c_i, d_i)$ in each $P_i$ where $c_i\in A$ and $d_i\in B$. As 
$\rank (M_G[\{c_i:i\in [t+1]\}\to \{d_i:i\in [t+1]\}])\le \rank (M_G[A\to B])=t$,
the matrix $M_G[\{c_i:i\in [t+1]\}\to \{d_i:i\in [t+1]\}]$ is linearly dependent. In particular, there is a non-empty subset $I$ of $[t+1]$ where the sum of $M_G[\{c_j\}\to \{d_i:i\in I\}]$ over $j\in I$ becomes a zero vector (as we are working on the binary field). We choose such a subset with minimum~$\abs{I}$. Without loss of generality, we may assume that $I=[q]$ for some $q\ge 2$. 

Now, we claim that for each $x\in [q]$, $c_x$ has an out-neighbor in $\{d_i:i\in [q]\setminus \{x\}\}$. Suppose that this is not true, that is, there is $x\in [q]$ where $c_x$ has no out-neighbor in $\{d_i:i\in [q]\setminus \{x\}\}$. This means that the row of $M_G[\{c_i:i\in [q]\}\to \{d_i:i\in [q]\}]$ indexed by $c_x$ does not affect on the sum on columns indexed by $\{d_i:i\in [q]\setminus \{x\}\}$. Therefore, the sum of $M_G[\{c_j\}\to \{d_i:i\in [q]\setminus \{x\}\}]$ over $j\in [q]\setminus \{x\}$ must be a zero vector. This contradicts the minimality of $I$.

We deduce that there exists a sequence $(i_1, i_2, \ldots, i_y)$ of at least two distinct elements in~$[q]$ with $i_{y+1}=i_1$ such that for each $j\in [y]$, there exists an edge from $c_{i_j}$ to $d_{i_{j+1}}$. For each $j\in [q]$, let $\ell_{i_j}$ be the length of the subpath of $P_{i_j}$ from $a_{i_j}$ to $c_{i_j}$. Observe that if $\ell_{i_j}\le \ell_{i_{j+1}}$, then there is a directed path of length at most $r$ from $a_{i_j}$ to $b_{i_{j+1}}$ by using the sub-path $a_{i_j}$ to $c_{i_j}$, then the edge to $d_{i_{j+1}}$ and then the sub-path to $b_{i_{j+1}}$, which contradicts the assumption that there is no edge from $a_{i_j}$ to $b_{i_{j+1}}$ in $G^r$. On the other hand, because of the cycle structure it is not possible that for all $j\in [x]$, $\ell_{i_j}>\ell_{i_{j+1}}$. So, we have a contradiction. 
 
\end{proof}

Note that the same argument holds for undirected graphs; if an undirected graph $G$ has rank-width $w$, then any power of $G$ has mim-width at most $w$. This extends the two arguments in~\cite{JaffkeKST2019} that any power of an undirected graph of tree-width $w-1$ has mim-width at most $w$, and any power of an undirected graph of clique-width $w$ has mim-width at most $w$, because such graphs have rank-width at most $w$~\cite{Oum05, Oum2008}.
 
 Next, we show that the $r$-th power of a digraph of bi-mim-width $w$ has bi-mim-width at most $rw$.  
 This will be used to prove Theorem~\ref{thm:xpalgorithmdistance}.
 
\begin{lemma}\label{lem:powerbimim}
Let $r$ and $w$ be positive integers. If $(\decT, \decf)$ is branch-decomposition of a digraph $G$ of bi-mim-width $w$, then it is a branch-decomposition of $G^r$ has bi-mim-width at most $rw$.
\end{lemma}
\begin{proof}
It is sufficient to prove that for every ordered vertex partition $(A, B)$ of $G$, we have $\nu(G^r[A\to B])\le r\nu (G[A\to B])$. Assume $\nu (G[A\to B])=t$ and suppose for contradiction that $\nu(G^r[A\to B])\ge rt+1$.    

Let $\{(a_i, b_i):i\in [rt+1]\}$ be an induced matching of $G^r[A\to B]$ with $\{a_i:i\in [rt+1]\}\subseteq A$. For each  $i\in [rt+1]$, let $P_i$ be a directed path of length at most $r$ from $a_i$ to $b_i$ in $G$. We choose an edge $(c_i, d_i)$ in each $P_i$ where $c_i\in A$ and $d_i\in B$. For each $i\in [rt+1]$, let $\ell_i$ be the length of the subpath of $P_i$ from $a_i$ to $c_i$. Observe that $0\le \ell_i\le r-1$.

By the pigeonhole principle, there exists a subset $I$ of $[rt+1]$ of size at least $t+1$ such that for all $i_1, i_2\in I$, $\ell_{i_1}=\ell_{i_2}$. Since $\nu (G[A\to B])=t$, there exist distinct integers $i_1, i_2\in I$ such that 
there is an edge from $c_{i_1}$ to $d_{i_2}$. Then there is a path of length at most $d$ from $a_{i_1}$ to $b_{i_2}$, contradicting the assumption that there is no edge from $a_{i_1}$ to $b_{i_2}$ in $G^r$. 
\end{proof}

\section{Classes of digraphs of bounded bi-mim-width}\label{sec:classes}

In this section, 
we present several digraph classes of bounded bi-mim-width, which are reflexive $H$-graphs (Proposition~\ref{prop:reflexiveH}), adjusted permutation digraphs (Proposition~\ref{prop:adjustedpermu}), 
adjusted rooted directed path digraphs (Proposition~\ref{prop:reflexivecontainmentH}), and nice $H$-convex graphs (Proposition~\ref{prop:reflexiveconvexH}).

We recall that 
a digraph $G$ is an \emph{intersection digraph} 
if there exists a family $\{(S_v, T_v):v\in V(G)\}$ of ordered pairs of sets, called a \emph{representation}, such that there is an edge from $v$ to $w$ in $G$ if and only if $S_v$ intersects $T_w$.

\subsection{$H$-digraphs and interval digraphs}

We define $H$-digraphs, which generalize interval digraphs.

\begin{definition}[$H$-digraph]
    Let $H$ be an undirected graph. 
    A digraph $G$ is an \emph{$H$-digraph} if there is a subdivision of $H$ and a family
    $\{(S_v, T_v):v\in V(G)\}$ of ordered pairs of connected subgraphs of $F$ such that $G$ is the intersection digraph with representation $\{(S_v, T_v):v\in V(G)\}$.
\end{definition}

\begin{definition}[Interval digraph]
    A digraph $G$ is an interval digraph if it is a $P_2$-digraph.
\end{definition}

We show that for fixed $H$, reflexive $H$-digraphs have bounded linear bi-mim-width.

\begin{proposition}\label{prop:reflexiveH}
Let $H$ be an undirected graph. Given a representation of a reflexive $H$-digraph $G$, one can construct in polynomial time a linear branch decomposition of $G$ of bi-mim-width at most $12\abs{E(H)}$.
\end{proposition}
\begin{proof}
Let $m:=\abs{E(H)}$. We may assume that $H$ is connected. If $H$ has no edge, then it is trivial. Thus, we may assume that $m\ge 1$. 

Let $G$ be a reflexive $H$-digraph, let $F$ be a subdivision of $H$, and let $\calM := \{(\sourceset_v, \targetset_v): v\in V(G)\}$ be a given reflexive $H$-digraph representation of $G$ with underlying graph $F$. For each $v\in V(G)$, choose a vertex $\alpha_v$ in $S_v\cap T_v$. We may assume that vertices in $(\alpha_v: v\in V(G))$ are pairwise distinct and they are not branching vertices, by subdividing $F$ more and changing $\calM$ accordingly, if necessary. 

We fix a branching vertex $r$ of $F$ and obtain a BFS ordering of $F$ starting from $r$. We denote by $v<_B w$ if $v$ appears before $w$ in the BFS ordering. We give a linear ordering $L$ of $G$ such that for all $v,w\in V(G)$,
if $\alpha_v<_B \alpha_w$, then $v$ appears before $w$ in $L$. This can be done in linear time.

We claim that $L$ has width at most $12m$. We choose a vertex $v$ of $G$ arbitrarily, and let $A$ be the set of vertices in $G$ that are $v$ or a vertex appearing before $v$ in $L$, and let $B:=V(G)\setminus A$. It suffices to show $\bimimval_G(A)\le 12m$. Let $A^*$ be the set of vertices of $F$ that are $\alpha_v$ or a vertex appearing before $\alpha_v$, and let $B^*:=V(F)\setminus A^*$. Let $\mathcal{P}$ be the set of paths in $F$ such that
\begin{itemize}
    \item for every $P\in \mathcal{P}$, $P$ is a subpath of some branching path of $F$ and it is a maximal path contained in one of $A^*$ and $B^*$, 
    \item $\bigcup_{P\in \mathcal{P}}V(P)=V(F)$.
\end{itemize}
Because of the property of a BFS ordering, it is easy to see that each branching path of $F$ is partitioned into at most $3$ vertex-disjoint paths in $\mathcal{P}$. Thus, we have $\abs{\mathcal{P}}\le 3m$. Note that two paths in $\mathcal{P}$ from two distinct branching paths may share an endpoint.

We first show that $\mim_G^+(A)\le 6m$.
Suppose for contradiction that 
$G[A\to B]$ contains an induced matching $M$ of size $6m+1$. 
By the pigeonhole principle, there is a subset $M_1=\{(x_i, y_i):i\in [3]\}$ of $M$ of size $3$ and a path $P$ in $\mathcal{P}$ such that for every $(x,y)\in M_1$, $S_{x}$ and $T_{y}$ meet on $P$. 
Let $p_1, p_2$ be the endpoints of $P$. 

Observe that $V(P)\subseteq A^*$ or $V(P)\subseteq B^*$. So, for each $i\in [3]$, it is not possible that $\alpha_{x_i}$ and $\alpha_{y_i}$ are both contained in $V(P)$.
It implies that each connected component of $(S_{x_i}\cup T_{y_i})\cap P$ contains an endpoint of $P$, as $S_{x_i}\cup T_{y_i}$ is connected. Therefore, there are at least two integers $j_1, j_2\in [3]$ and a connected component $C_1$ of $(S_{x_{j_1}}\cup T_{y_{j_1}})\cap P$ and a connected component $C_2$ of $(S_{x_{j_2}}\cup T_{y_{j_2}})\cap P$ so that 
\begin{itemize}
    \item $C_1$ and $C_2$ contain the same endpoint of $P$, and 
    \item for each $i\in [2]$, $C_i$ contains a vertex of $S_{x_{j_i}}$ and a vertex of $T_{y_{j_i}}$.
\end{itemize} 
However, it implies that $(x_{j_1}, y_{j_2})$ or $(x_{j_2}, y_{j_1})$ is an edge, a contradiction.

We deduce that $\mim_G^+(A)\le 6m$. 
 By a symmetric argument, we get $\mim_G^-(A)\le 6m$.
Therefore, we have $\bimimval_G(A)\le 12m$, as required.
\end{proof}

We obtain a better bound for reflexive interval digraphs.

\begin{proposition}\label{prop:reflexiveinterval}
Given a reflexive interval digraph, one can output a linear branch decomposition of bi-mim-width at most $2$ in polynomial time. 
\end{proposition}
\begin{proof}
Let $G$ be a given reflexive interval digraph. By M{\"u}ller's recognition algorithm for interval digraphs~\cite{Muller1997}, one can out its representation in polynomial time.

    Now, we follow the proof of Proposition~\ref{prop:reflexiveH}. In this case, $\mathcal{P}$ consists of exactly two paths, one induced by $A^*$ and the other induced by $B^*$. Because of it, it is not difficult to observe that $\mim_G^+(A)\le 1$ and $\mim_G^-(A)\le 1$ (similar to interval graphs). Thus, it has linear bi-mim-width at most $2$.
\end{proof}

\begin{figure}
  \centering
  \begin{tikzpicture}[scale=0.5, decoration={
    markings,
    mark=at position 0.5 with {\arrow{>}}}]
  \tikzstyle{w}=[circle,draw,fill=black!50,inner sep=0pt,minimum width=3pt]

 \foreach \y in {0, 1, 2, 3, 4,5}{
	\draw[postaction={decorate}] (2, \y+1)-- (2, \y);
	\draw[postaction={decorate}] (4, \y+1)--(4, \y);
	\draw[postaction={decorate}] (6, \y+1)--(6, \y);

	\draw[postaction={decorate}] (1, \y) -- (1, \y+1);
	\draw[postaction={decorate}] (3, \y)--(3, \y+1);
	\draw[postaction={decorate}] (5, \y)--(5, \y+1);
	\draw[postaction={decorate}] (7, \y)--(7, \y+1);

}


 \foreach \y in {0, 1, 2, 3, 4,5, 6}{
 \foreach \x in {1,3,5}{
	\draw[postaction={decorate}] (\x, \y)-- (\x+1, \y);
}
 \foreach \x in {2,4,6}{
	\draw[postaction={decorate}] (\x+1, \y)-- (\x, \y);
}
}

\end{tikzpicture} \caption{The interval digraph $G_7$ in Proposition~\ref{prop:intunbounded}. The left bottom vertex is $v_{1,1}$ and its right-hand vertex is $v_{2,1}$ and so on. }\label{fig:interval}
\end{figure}
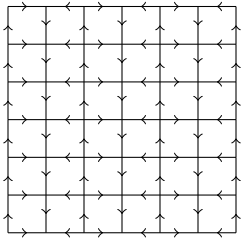

\begin{proposition}\label{prop:intunbounded}
Interval digraphs have unbounded bi-mim-width.
\end{proposition}
\begin{proof}
We will construct some orientation of the $(n\times n)$-grid as an interval digraph.


For $i,j\in [n]$, we construct $S_{v_{i,j}}$ and $T_{v_{i,j}}$ as follows.
For every odd integer $i$, we set \begin{itemize}
    \item $S_{v_{i,j}}:=[2(n+1)j+2i-1, 2(n+1)j+2i+1]$ and,
    \item $T_{v_{i,j}}:=[2(n+1)(j-1)+2i, 2(n+1)(j-1)+2i]$,
\end{itemize} 
and for every even integer $i$, we set
\begin{itemize}
    \item $S_{v_{i,j}}:=[2(n+1)(j-1)+2i, 2(n+1)(j-1)+2i]$ and
    \item $T_{v_{i,j}}:=[2(n+1)j+2i-1, 2(n+1)(j-1)+2i+1]$.
\end{itemize}
Let $G_n$ be the intersection digraph on the vertex set $\{v_{i,j}:i,j\in [n]\}$  0with representation $\{(S_{v_{i,j}}, T_{v_{i,j}}):i,j\in [n]\}$. 
Observe that for $i,j\in [n-1]$, if $i$ is odd, then $(v_{i,j}, v_{i+1, j})$ and $(v_{i,j}, v_{i,j+1})$ are edges, and if $i$ is even, then $(v_{i+1, j}, v_{i,j})$ and $(v_{i,j+1}, v_{i,j})$ are edges. See Figure~\ref{fig:interval} for an illustration. 

This is an orientation of the $(n\times n)$-grid. It is known that the $(n\times n)$-grid has mim-width at least $n/3$; see~\cite[Theorem 4.3.10]{VatshelleThesis}. By Lemma~\ref{lem:underlying}, it has bi-mim-width at least $n/3$.
\end{proof}

\subsection{Permutation digraphs}

Permutation digraphs are directed analogues of permutation graphs.

\begin{definition}[Permutation digraph]
    A digraph $G$ is a \emph{ permutation digraph} if there is a family
    $\{(S_v, T_v):v\in V(G)\}$ of line segments whose endpoints lie on two parallel lines $\Lambda_1$ and $\Lambda_2$ where $G$ is the intersection digraph with representation $\{(S_v, T_v):v\in V(G)\}$. A permutation digraph is \emph{adjusted} if $S_v$ and $T_v$ have a common endpoint in $\Lambda_1$ for all $v\in V(G)$.
\end{definition}

We observe that interval digraphs are permutation digraphs. When we have a representation $\{(S_v, T_v):v\in V(G)\}$ for an interval digraph on a line~$\Lambda=\{x:x\in \mathbb{R}\}$, we consider two copies $\Lambda_1=\{(x,0):x\in \mathbb{R}\}$ and $\Lambda_2=\{(x,1):x\in \mathbb{R}\}$ of~$\Lambda$, and then for each $S_v=[a_v, b_v]$, we make $S_v^*$ as the line segment linking $(a_v, 0)$ and $(b_v, 1)$, and for each $T_v=[c_v,d_v]$, we make $T_v^*$ as the line segment linking $(c_v, 1)$ and $(d_v, 0)$. One can verify that $\{(S_v^*, T_v^*):v\in V(G)\}$ is a representation of the same digraph. Because of this and by Proposition~\ref{prop:intunbounded}, permutation digraphs also have unbounded bi-mim-width.

We show that adjusted permutation digraphs have linear bi-mim-width at most $4$. Note that an adjusted permutation graph is reflexive, but we were not able to show that all reflexive permutation digraphs have bounded bi-mim-width. It remains open whether there is a constant bound on (linear) bi-mim-width of reflexive permutation digraphs.

\begin{proposition}\label{prop:adjustedpermu}
Given a representation of an adjusted permutation digraph $G$, one can construct in polynomial time a linear branch decomposition of $G$ of bi-mim-width at most $4$.
\end{proposition}
\begin{proof}
Let $\Lambda_1:=\{(x,0):x\in \mathbb{R}\}$ and $\Lambda_2:=\{(x,1):x\in \mathbb{R}\}$ be two lines.
Let $G$ be a given adjusted permutation digraph with its representation $\{(S_v, T_v):v\in V(G)\}$ where $S_v$ and $T_v$ are line segments whose endpoints lie on $\Lambda_1$ and $\Lambda_2$ and they have a common endpoint in $\Lambda_1$, say $(\alpha_v, 0)$. 
For each $v\in V(G)$, let $(\beta_v, 1)$ be the endpoint of $S_v$ in $\Lambda_2$ and $(\gamma_v, 1)$ be the endpoint of $T_v$ in $\Lambda_2$.
 
We give a linear ordering $L$ of $G$ such that for all $v,w\in V(G)$,
if $\alpha_v< \alpha_w$, then $v$ appears before $w$ in $L$. This can be done in linear time.

We claim that $L$ has bi-mim-width at most $4$. We choose a vertex $v$ of $G$ arbitrarily, and let $A$ be the set of vertices in $G$ that are $v$ or a vertex appearing before $v$ in $L$, and let $B:=V(G)\setminus A$. 

We verify that $\mim_G^+(A)\le 2$. Suppose for contradiction that $G[A\to B]$ has an induced matching $\{(v_i,w_i):i\in [3]\}$ with $v_1, v_2, v_3\in A$. Without loss of generality, we assume that $\alpha_{v_1}\le \alpha_{v_2}\le \alpha_{v_3}$. Observe that $\alpha_{w_1}, \alpha_{w_2}>\alpha_{v_3}$ and $\alpha_{w_3}\ge \alpha_{v_3}$.
Let $w\in \{w_i:i\in [3]\}$ such that $\abs{\alpha_w-\alpha_{v_3}}$ is minimum.
We distinguish the case depending on whether $w=v_3$ or not.

\medskip
(Case 1. $\alpha_w\neq \alpha_{v_3}$.) First assume that $w=w_3$. As $\{(v_i, w_i):i\in [3]\}$ is an induced matching, we have $\beta_{v_1}<\gamma_{w_3}$ and and $\beta_{v_3}<\gamma_{w_1}$. This implies that $S_{v_1}$ and $T_{w_1}$ do not meet, a contradiction. Now, assume that $w\neq w_3$ and thus, $S_{v_3}$ and $T_w$ do not meet. Let $v^*$ be the vertex where $(v^*, w)$ is in the induced matching. As $\{(v_i, w_i):i\in [3]\}$ is an induced matching, $\beta_{v^*}<\gamma_{w_3}$ and $\gamma_w>\beta_{v_3}$. Then 
$S_{v^*}$ cannot meet $T_{w}$, a contradiction.  

\medskip
(Case 2. $\alpha_w=\alpha_{v_3}$.)  In this case, $w$ must be $w_3$. If $\gamma_{w_3}\le \beta_{v_3}$, then $S_{v_1}$ and $T_{w_1}$ cannot meet. So, $\gamma_{w_3}>\beta_{v_3}$, and it is not difficult to verify that $\beta_{v_3}<\beta_{v_1}, \beta_{v_2}, \gamma_{w_1}, \gamma_{w_2}<\gamma_{w_3}$, otherwise, $\{(v_i,w_i):i\in [3]\}$ cannot be an induced matching. If $\beta_{v_1}\le \beta_{v_2}$, then $T_{w_1}$ has to meet $S_{v_2}$, a contradiction. Similarly, if $\beta_{v_2}< \beta_{v_1}$, then $T_{w_2}$ has to meet $S_{v_1}$, a contradiction.

\medskip
It shows that $\mim_G^+(A)\le 2$, as claimed. By a symmetric argument, we have $\mim_G^-(A)\le 2$ and $\bimimval_G(A)\le 4$.

We conclude that $L$ has bi-mim-width at most $4$.
\end{proof}

\subsection{Rooted directed path digraphs}

A \emph{directed tree} is an orientation of a tree, and it is \emph{rooted} if there is a root node $r$ such that every vertex is reachable from $r$ by a directed path. 
Gavril~\cite{Gavril1975} introduced the class of \emph{rooted directed path graphs}, that are intersection graphs of directed paths in a rooted directed tree. We introduce its directed analogue.

\begin{definition}[Rooted directed path digraph]
    A digraph $G$ is a \emph{rooted directed path digraph} if there is a rooted directed tree $T$ and a family
    $\{(S_v, T_v):v\in V(G)\}$ of pairs of directed paths in $T$ where $G$ is the intersection digraph with representation $\{(S_v, T_v):v\in V(G)\}$. A directed rooted path digraph is \emph{adjusted} if for every $v\in V(G)$, the endpoint of $S_v$ that is farther from the root is the same as the endpoint of $T_v$ that is farther from the root.
\end{definition}

Clearly, interval digraphs are rooted directed path digraphs, and therefore, rooted directed path digraphs have unbounded bi-mim-width. We prove that adjusted rooted directed path digraphs have bounded bi-mim-width and have unbounded linear bi-mim-width.

\begin{proposition}\label{prop:reflexivecontainmentH}
Given a representation of an adjusted rooted directed path digraph $G$, one can construct in polynomial time a branch decomposition of $G$ of bi-mim-width at most $2$. 
Adjusted rooted directed path digraphs have unbounded linear bi-mim-width.
\end{proposition}
\begin{proof}
Let $F$ be a rooted directed tree with root node $r$. Let $G$ be an adjusted rooted directed path digraph with representation $\mathcal{M}=\{(\sourceset_v, \targetset_v): v\in V(G)\}$ where $S_v$ and $T_v$ are directed paths in $F$.

Now, we modify $(F, \mathcal{M})$ into $(F^*, \mathcal{M}^*=\{(\sourceset_v^*, \targetset_v^*): v\in V(G)\})$ so that $F^*$ is a rooted directed tree and $\mathcal{M}^*$ is an adjusted rooted directed path digraph representation of $G$ consisting of directed paths in $F^*$ with additional conditions that 
\begin{itemize}
    \item every internal node of $F^*$ has out-degree at most $2$, 
    \item for every vertex $v\in V(G)$, the endpoint $\alpha_v$ of $S^*_v$ that is farther from the root of $F^*$ has out-degree at most $1$ in $F^*$, and nodes in $(\alpha_v:v\in V(G))$ are pairwise distinct.
\end{itemize}

Let $(F_1, \mathcal{M}_1):=(F, \mathcal{M})$. For every $v\in V(G)$, let $\alpha_v^1$ be the endpoint of $S_v$ that is farther from the root in $F$.
We  recursively construct $(F_i, \mathcal{M}_i=\{(\sourceset_v^i, \targetset_v^i): v\in V(G)\}), (\alpha_v^i:v\in V(G)))$ until we get the desired conditions.
\begin{itemize}
    \item Assume that $F_i$ has a node $t$ of out-degree at least $3$. Let $t_1, \ldots, t_x$ be the out-neighbors of $t$ in $F_i$. Now, we remove all edges between $t$ and $\{t_1, \ldots, t_x\}$, and then add a directed path $p_1p_2 \cdots p_x$ and add an edge $(t, p_1)$ and edges $(p_i, t_i)$ for all $i\in [x]$. Now, for every path $P$ containing $tt_j$ for some $j$ in the representation $\mathcal{M}_i$ we replace $tt_j$ with a path $tp_1p_2 \cdots p_jt_j$. The other paths do not change. The resulting rooted directed tree and representation are $F_{i+1}$ and $\mathcal{M}_{i+1}$, respectively.
    \item Assume that $\alpha_v^i=\alpha_w^i=t$ for some distinct vertices $v$ and $w$ in $G$. Let $q$ be an out-neighbor of $t$ in $F_i$ if one exists, and otherwise we attach a new node $q$ and add an edge $(t,q)$. We replace $(t,q)$ with a directed path $tt'q$ to obtain $F_{i+1}$. For every path in  $\mathcal{M}_i$ containing $(t,q)$ we replace $(t,q)$ with a directed path $tt'q$, and assign $\alpha_v^{i+1}:=t'$.
    For every path in $\{S^i_v, T^i_v\}$, we extend it by adding $tt'$. The other paths do not change. The resulting  representation is $\mathcal{M}_{i+1}$. 
    
    \item Assume that $\alpha_v^i=t$ for some vertex $v$ in $G$, where $t$ is a node of out-degree at least $2$. Let $q$ be an out-neighbor of $t$ in $F_i$. We replace $(t,q)$ with a directed path $tt'q$ to obtain $F_{i+1}$. For every path in $\{S^i_v, T^i_v\}$, we extend it by adding $tt'$, and we assign $\alpha_v^{i+1}:=t'$. The other paths do not change. The resulting  representation is $\mathcal{M}_{i+1}$.
\end{itemize} 
In each iteration, either the number of nodes of out-degree at least $3$ decreases, or the number of pairs of vertices $v$ and $w$ for which $\alpha_v^i=\alpha_w^i$ decreases, or the number of vertices $v$ for which $\alpha_v^i$ is a node of out-degree at least $3$ decreases. Also, the process in each iteration preserve that the representation is an adjusted rooted directed path digraph representation. Thus, at the end, we obtain an adjusted rooted directed path digraph representation $(F^*, \mathcal{M}^*)$ with the set of common endpoints $(\alpha_v^*:v\in V(G))$, as desired.
Let~$U$ be the underlying undirected tree of $F^*$.

For each $v\in V(G)$, we add a new node $\beta_v$ and add an edge $\beta_v\alpha_v^*$ to $U$, and obtain a tree $U^*$. Now, let $e$ be an edge of $U^*$, and let $U_X$ and $U_Y$ be the connected component of $U^* - e$ where $U_Y$ contains the root of $F^*$.
    Let $(X_e, Y_e)$ be the cut of $G$ where
    $X_e$ is the set of vertices $v$ for which $\alpha_v^*$ is in $U_X$, and $Y_e=V(G)\setminus X_e$.

We claim that $\mim_G^+(X_e)\le 1$. Suppose for contradiction that there is an induced matching $\{(v_i, w_i):i\in [2]\}$ with $v_1, v_2\in X_e$. 
Let $Q$ be the path in $U^*$ from the root to $e$. 

Assume that the distance from $e$ to $T_{w_1}$ in $U^*$ is at most the distance from $e$ to $T_{w_2}$. Since $S_{v_1}$ and $S_{v_2}$ are directed paths, $T_{w_1}$ and $T_{w_2}$ contain a vertex of $Q$. It means that $S_{v_2}$ has to meet $T_{w_1}$, which is a contradiction. If the dsitance from $e$ to $T_{w_2}$ in $U^*$ is at most the distance from $e$ to $T_{w_1}$, then $S_{v_1}$ has to meet $T_{w_2}$, a contradiction.

Thus, we have $\mim_G^+(X_e)\le 1$, and by symmetry, we can show that $\mim_G^-(X_e)\le 1$. It implies that $\bimimval_G(X_e)\le 2$. By smoothing degree $2$ nodes of $U^*$, we obtain a branch-decomposition of bi-mim-width at most $2$.

Now, we argue that 
adjusted rooted directed path digraphs have unbounded linear bi-mim-width. It is well known that trees are rooted directed path graphs. Also, H{\o}gemo, Telle, and V{\aa}gset~\cite{HogemoTV2019} proved that trees have unbounded linear mim-width. As trees are rooted directed path graphs, we can obtain the biorientations of trees as adjusted rooted directed path digraphs, where $S_v=T_v$ for all $v\in V(G)$. As trees have unbounded linear mim-width, adjusted rooted directed path digraphs have unbounded linear bi-mim-width.
\end{proof}

\subsection{$H$-convex digraphs}

As a generalization of convex graphs, Brettell, Munaro, and Paulusma~\cite{BrettellMP2020} introduced $H$-convex graphs. Note that they defined $\mathcal{H}$-convex graphs, rather than $H$-convex graphs, where $\mathcal{H}$ is a family of graphs. However, they mostly considered $\mathcal{H}$ as the set of all subdivisions of a fixed graph $H$, so it seems better to call them simply $H$-convex graphs.
We generalize this notion to $H$-convex digraphs. 

\begin{definition}[$H$-convex digraph]
    Let $H$ be an undirected graph. 
    A bipartite digraph $G$ with bipartition $(A, B)$ is an \emph{$H$-convex digraph} if there is a subdivision $F$ of $H$ with $V(F)=A$ such that for every vertex $b$ of $B$, each of $N^+_G(v)$ and $N^-_G(v)$ induces a connected subgraph of $F$.
    An $H$-convex digraph is \emph{nice} if for every vertex $b$ of $B$, there is a bi-directed edge between $b$ and some vertex of $A$.
\end{definition}

In principle, $H$-convex digraphs are $H$-digraphs are closely related, as $\{(N_G^+(v), N_G^-(v)):v\in B\}$ can be seen as an $H$-digraph representation on $F$.
We prove that nice $H$-convex digraphs have linear bi-mim-width at most $12\abs{E(H)}$. 

\begin{proposition}\label{prop:reflexiveconvexH}
Let $H$ be an undirected graph. Given a nice $H$-convex digraph $G$ with its bipartition $(A,B)$, one can construct in polynomial time a linear branch decomposition of $G$ of bi-mim-width at most $12\abs{E(H)}$.
\end{proposition}
\begin{proof}
Let $m:=\abs{E(H)}$. Let $G$ be a reflexive $H$-convex digraph with bipartition $(X, Y)$. Let $F$ be a subdivision of $H$ with $X=V(F)$ and let $\mathcal{M}=\{(N^+_G(y), N^-_G(y)):y\in Y\}$.
For each $y\in Y$, we choose $\alpha_y\in N^+_G(y)\cap N^-_G(y)$. By adding more vertices in $X$ if necessary (corresponding to a subdivision of $F$), we may assume that vertices in $(\alpha_y:y\in Y)$ are pairwise distinct and they are not branching vertices of $F$. This assumption is possible because of Lemma~\ref{lem:subdigraph}.

We fix a vertex $r$ of $F$ and obtain a BFS ordering of $F$ starting from $r$. We denote by $v<_B w$ if $v$ appears before $w$ in the BFS ordering. This gives an ordering of $V(F)=X$.
Now, we extend it to an ordering $L$ of $V(G)$ by, for each $y\in Y$, putting $y$ just after $\alpha_y$.

We claim that $L$ has width at most $12m$. We choose a vertex $v$ of $G$ arbitrarily, and let $A$ be the set of vertices in $G$ that are $v$ or a vertex appearing before $v$ in $L$, and let $B:=V(G)\setminus A$. It is sufficient to show that $\mim_G^+(A)\le 6m$. By symmetry, we will get $\mim_G^-(A)\le 6m$.

Let $A_X:=A\cap X$, $A_Y:=A\cap Y$, $B_X:=B\cap X$, and $B_Y:=B\cap Y$. We first show that 
$\nu(G[A_X\to B_Y])\le 4m$. 

Suppose for contradiction that 
$G[A_X\to B_Y]$ contains an induced matching $M$ of size $4m+1$. 
As defined in Proposition~\ref{prop:reflexiveH},
we define $\mathcal{P}$ as the set of paths in $F$ such that
\begin{itemize}
    \item for every $P\in \mathcal{P}$, $P$ is a subpath of some branching path of $F$ and it is a maximal path contained in one of $A_X$ and $B_X$, 
    \item $\bigcup_{P\in \mathcal{P}}V(P)=V(F)$.
\end{itemize}
Note that each branching path contains at most two paths in $\mathcal{P}$ that are contained in $A_X$.
By the pigeonhole principle, there is a subset $M_1=\{(x_i, y_i):i\in [3]\}$ of $M$ of size $3$ and a path $P\in \mathcal{P}$ with $V(P)\subseteq A_X$ such that for every $(x,y)\in M_1$, $x$ is in $P$. Observe that $\alpha_{y_i}$ is not in $P$ for every $i\in [3]$. 

Let $p_1$ and $p_2$ be the endpoints of $P$. Observe that for each $i\in [3]$, $N^-_G(y_i)$ contains either the subpath of $P$ from $p_1$ to $x_i$ or the subpath of $P$ from $x_1$ to $p_2$. Without loss of generality, we may assume that for each $i\in [2]$, $N^-_G(y_i)$ contains the subpath of $P$ from $p_1$ to $x_i$.
But this implies that $(x_1, y_2)$ or $(x_2, y_1)$ is an edge, contradicting the assumption that $M$ is an induced matching.

We deduce that $\nu(G[A_X\to B_Y])\le 4m$. Note that each branching path contains at most one path in $\mathcal{P}$ that is contained in $B_x$. Thus, by a similar argument, we can show that $\nu(G[A_Y\to B_X])\le 2m$. So, we have $\mim_G^+(A)=\nu(G[A\to B])\le 6m$. By symmetry, we get $\mim_G^-(A)\le 6m$ as well, and these imply that $\bimimval_G(A)\le 12m$, as required.
\end{proof}

\begin{proposition}\label{prop:convexunbounded}
$P_2$-convex digraphs have unbounded bi-mim-width.
\end{proposition}
\begin{proof}
We recall the interval digraph representation of an orientation of the $(n\times n)$-grid, given in Proposition~\ref{prop:intunbounded}.
For $i,j\in [n]$, we construct $S_{i,j}$ and $T_{i,j}$ as follows.
For every odd integer $i$, we set \begin{itemize}
    \item $S_{i,j}:=[2(n+1)j+2i-1, 2(n+1)j+2i+1]$ and,
    \item $T_{i,j}:=[2(n+1)(j-1)+2i, 2(n+1)(j-1)+2i]$,
\end{itemize} 
and for every even integer $i$, we set
\begin{itemize}
    \item $S_{i,j}:=[2(n+1)(j-1)+2i, 2(n+1)(j-1)+2i]$ and
    \item $T_{i,j}:=[2(n+1)j+2i-1, 2(n+1)(j-1)+2i+1]$.
\end{itemize}

Now, we create a bipartite digraph $G_n$ with bipartition $(A, B=\{v_{i,j}:i,j\in [n]\})$ such that  
for $i,j\in [n]$, 
\begin{itemize}
    \item if $i$ is odd, then $N^+_G(v_{i,j})=S_{i,j}$ and $N^-_G(v_{i,j})=T_{i,j}$, 
    \item if $i$ is even, then 
    $N^+_G(v_{i,j})=T_{i,j}$ and $N^-_G(v_{i,j})=S_{i,j}$.
\end{itemize}
It is not difficult to verify that $G_n$ contains an orientation of the $1$-subdivision of the $(n\times n)$-grid as an induced subdigraph. By modifying the proof for the fact that the $(n\times n)$-grid has mim-width at least $n/3$, it is straightforward to show that the $1$-subdivision of the $(n\times n)$-grid has mim-width at least $n/3$. Therefore, by Lemma~\ref{lem:underlying}, $G_n$ has bi-mim-width at least $n/3$.
\end{proof}

\section{Algorithmic applications}\label{sec:application}
In this section we give the algorithmic applications of the width measure bi-mim-width.
In particular, we show that all directed locally checkable vertex subset and all directed locally checkable vertex partitioning problems can be solved in $\XP$ time paramterized by the bi-mim-width of a given branch decomposition of the input digraph.
We do so by adapting the framework of the $d$-neighborhood equivalence relation introduced by Bui-Xuan et al.~\cite{Bui-XuanTV13} to digraphs.
For an undirected graph $G$, given a set $A \subseteq V(G)$, two subsets $X$ and $Y$ of $A$ are $d$-neighborhood equivalent w.r.t.~$A$ if the intersection of the neighborhood of each vertex in $\bar{A}$ with $X$ and $Y$ have the same size, when counting up to $d$.

In the adaptation of this concept to digraphs,
we essentially take the cartesian product of the $d$-neighborhood equivalences given by the edges going from $A$ to $\bar{A}$ and the edges going from $\bar{A}$ to $A$.
In the resulting $d$-bi-neighborhood equivalence relation of a set $A$ of vertices in some digraph $G$,
two subsets $X$ and $Y$ of $A$ are $d$-bi-neighborhood equivalent, 
if for each vertex of $\bar{A}$, the number of both its in-neighbors in $X$ and the number of its out-neighbors in $X$ are equal to the number of its in- and out-neighbors in $Y$, respectively, when counted up to $d$.
We show that this notion allows us to lift the frameworks presented in~\cite{Bui-XuanTV13} to the realm of digraphs and prove the aforementioned results.

The rest of this section is organized as follows.
In \cref{sec:bi-nb-equiv}, we formally define the $d$-bi-neighborhood equivalence relation and show how to efficiently compute descriptions of its equivalence classes. 
In \cref{sec:alg:sigma-rho} we give the algorithms for generalized directed domination problems and in \cref{sec:alg:lcvp} for the directed vertex partitioning problems.
We discuss how to use these algorithms to solve distance-$r$ versions of LVCS and LCVP problems in \cref{sec:alg:distance-r}.

\subsection{$d$-Bi-neighhorhood-equivalence}\label{sec:bi-nb-equiv}
We now present the central notion that is used in our algorithms, the
$d$-bi-neighborhood equivalence relation which we introduced informally earlier.
The reason why it compares sizes of the intersection of a neighborhood with subsets only up to some fixed integer $d$ is as follows.
The subsets of natural numbers that characterize locally checkable vertex subset/partitioning problems can be fully characterized when counting in- and out-neighbors up to some constant $d$, depending on the described problem.
Therefore, if a vertex $v$ has more than $d$ for instance out-neighbors in two sets $X$ and $Y$, then these two sets look the same to $v$ in terms of its out-neighborhood.

In the following definition, we present the $d$-in-neighborhood equivalence relation and the $d$-out-neighborhood equivalence relation separated before combining them to the $d$-bi-neighborhood equivalence relation, since in some proofs it is convenient to only consider the edges going in one direction.
\begin{definition}
	Let $d \in \bN$.
	Let $G$ be a digraph and $A \subseteq V(G)$.
	For two sets $X, Y \subseteq A$, we say that $X$ and $Y$ are \emph{$d$-out-neighborhood equivalent},
	written $X \equiv^+_{d, A} Y$, if\footnote{%
	Since the definition is given in terms of vertices from $\bar{A}$, 
	we consider the directions of the edges in reverse,
	i.e., we consider $N^-(v)$ for $v \in \bar{A}$ 
	when defining $\equiv^+$.}
	\begin{align*}
		\forall u \in V(G) \setminus A\colon \min\{d, \card{N^-(u) \cap X}\} = \min\{d, \card{N^-(u) \cap Y}\}.
	\end{align*}
	Similarly, we say that $X$ and $Y$ are \emph{$d$-in-neighborhood equivalent}, written $X \equiv^-_{d, A} Y$, if 
	\begin{align*}
		\forall u \in V(G) \setminus A\colon \min\{d, \card{N^+(u) \cap X}\} = \min\{d, \card{N^+(u) \cap Y}\}.
	\end{align*}
	If $X \equiv^+_{d, A} Y$ and $X \equiv^-_{d, A} Y$
	then we say that $X$ and $Y$ are $d$-bi-neighborhood equivalent and write
	$X \equiv^{\pm}_{d, A} Y$.
\end{definition}

The runtime of the algorithms in this section crucially depend on the number of equivalence classes of the $d$-bi-neighborhood equivalence relations associated with cuts induced by a branch decomposition of the input graph.
For $d, G,$ and $A$ as in the previous definition, we denote by $\nec(\equiv^\pm_{d, A})$ the number of equivalence classes of $\equiv^\pm_{d, A}$. 
If $(\decT, \decf)$ is a branch decomposition of $G$, we let
\begin{align*}
    \nec_d(\decT, \decf) = \max\nolimits_{t \in V(T)} \max \{\nec(\equiv^\pm_{d, V_t}), \nec(\equiv^\pm_{d, \bar{V_t}})\}.
\end{align*}

\subsubsection*{Descriptions of equivalence classes of $\equiv_{d, A}^\pm$}
Since $\equiv^\pm_{d, A}$ is an equivalence relation over subsets of $A$, 
we cannot trivially enumerate all its equivalence classes without risking an exponential running time.
We now show that we can enumerate the equivalence classes with a relatively small overhead depending polynomially on $n$, $d$, and $\log(\nec(\equiv^\pm_{d, A}))$. 
This enumeration is based on pairs of vectors called \emph{$d$-bi-neighborhoods} of a subset $X$ of $A$, one that describes the in-neighborhood of vertices in $\bar{A}$ intersected with $X$, and one for the out-neighborhood.
\begin{definition}
	Let $G$ be a digraph, $X \subseteq A \subseteq V(G)$, and $d \in \bN$.
	The \emph{$d$-out-neighborhood} of $X$, denoted by $\nbhv_{d, A}^+(X)$ is a
	vector in $\{0, 1, \ldots, d\}^{\bar{A}}$, 
	which stores for every vertex $v \in \bar{A}$
	the minimum between $d$ and the number of in-neighbors of $v$ in $X$. 
	Formally,
	\begin{align*}
		\nbhv_{d, A}^+(X) = (\min\{d, \card{N^-(v) \cap X}\})_{v \in \bar{A}}.
	\end{align*}
	Similarly, the \emph{$d$-in-neighborhood}, denoted by $\nbhv_{d, A}^-(X)$, is the vector
	\begin{align*}
		\nbhv_{d, A}^-(X) = (\min\{d, \card{N^+(v) \cap X}\})_{v \in \bar{A}}.
	\end{align*}
	We refer to the pair $(\nbhv^+_{d, A}(X), \nbhv^-_{d, A}(X))$ as the \emph{$d$-bi-neighborhood} $\nbhv^\pm_{d, A}(X)$;
	and we denote the set of all $d$-bi-neighborhoods as $\calnbhv^\pm_{d, A}$.
\end{definition}

\begin{observation}\label{obs:d-equiv:nbh}
	Let $G$ be a digraph and $X, Y \subseteq A \subseteq V(G)$.
	Then, $X \equiv_{d, A}^\pm Y$ if and only if $\nbhv^\pm_{d, A}(X) = \nbhv^\pm_{d, A}(Y)$.
\end{observation}

By the previous observation, there is a natural bijection between the $d$-bi-neighborhoods 
and the equivalence classes of $\equiv_{d, A}^\pm$.
In our algorithm we will therefore use the $d$-bi-neighborhoods
as descriptions for the equivalence classes of $\equiv^\pm_{d, A}$.
We now show that we can efficiently enumerate them.
While the ideas are parallel to the algorithm presented in~\cite{Bui-XuanTV13},
we work directly with the $d$-bi-neighborhoods rather than with representatives
to streamline the presentation.
\begin{lemma}\label{lem:enum:desc}
	Let $G$ be a digraph on $n$ vertices, $A \subseteq V(G)$, and $d \in \bN$.
	There is an algorithm that enumerates all members of $\calnbhv^\pm_{d, A}$ in time 
	$\calO(\nec(\equiv^\pm_{d, A})\log\nec(\equiv^\pm_{d, A})\cdot dn^2)$.
	Furthermore, for each $Y \in \calnbhv^\pm_{d, A}$,
	the algorithm can provide some $X \subseteq A$ with $\nbhv^\pm_{d, A}(X) = Y$.
\end{lemma}
\begin{proof}
	We describe the procedure to enumerate $\calnbhv^\pm_{d, A}$ in \cref{alg:d-bi-nbh}.
	\begin{algorithm}[tb]
		\SetKwInOut{Input}{Input}
		\SetKwInOut{Output}{Output}
		\Input{Digraph $G$, $A \subseteq V(G)$, and $d \in \bN$.}
		\Output{$\calnbhv^\pm_{d, A}$}
		Let $\ttN$ be an empty balanced binary search tree\; \tcc{{\small To enable binary search in $\ttN$ 
			we assume a fixed ordering on $V(G)$; 
			the vectors in $\bN^{\bar{A}}$ can then be compared lexicographically.}}
		Let $\calB = \{\emptyset\}$\;
		\While{$\calB \neq \emptyset$}{
			Let $\calB' = \emptyset$\;
			\ForEach{$R \in \calB$}{
				\For{$v \in V(G)$}{
					Let $R' = R \cup \{v\}$\;
					\lIf{$\nbhv^\pm_{d, A}(R') \notin \ttN$}{add $\nbhv^\pm_{d, A}(R')$ to $\ttN$ and $R'$ to $\calB'$\label{line:alg:d-bi-nbh-add}}
				}
			}
			Let $\calB = \calB'$\;
		}
		\Return{\ttN}
		\caption{Algorithm to compute all $d$-bi-neighborhoods.}
		\label{alg:d-bi-nbh}
	\end{algorithm}
	
	Let us argue that this algorithm is correct.
	It is easy to observe that $\ttN$ only contains pairwise distinct $d$-bi-neighborhoods.
	Suppose for a contradiction that there is a set $X \subseteq A$ such that $\nbhv^\pm_{d, A}(X) \notin \ttN$,
	and assume wlog.\ that $X$ is a minimal subset of vertices whose $d$-bi-neighborhood is not contained in $\ttN$.
	Let $u \in X$.
	We know that for all $Y \in \ttN$, 
	$\nbhv^\pm_{d, A}(X \setminus \{u\}) \neq Y$,
	for otherwise, $\nbhv^\pm_{d, A}(X)$ would have been added to $\ttN$.
	But this contradicts the minimality of $X$.
	
	\cref{alg:d-bi-nbh} can easily be modified to satisfy the second claim of the lemma:
	In line~\ref{line:alg:d-bi-nbh-add}, instead of adding only $\nbhv^\pm_{d, A}(R')$ to $\ttN$,
	we may add the pair $(R', \nbhv^\pm_{d, A}(R'))$.
	
	We analyze the runtime as follows.
	For each $d$-bi-neighborhood that is added to $\ttN$,
	we test for at most $n$ additional sets whether their $d$-bi-neighborhoods need to be added to $\ttN$ or not.
	Together with \cref{obs:d-equiv:nbh},
	this implies that we test for at most $\calO(\nec(\equiv^\pm_{d, A})\cdot n)$
	sets whether they should be added to $\ttN$ or not.
	Computing the $d$-bi-neighborhood of a set $X \subseteq A$ can be done in time $\calO(d\cdot n)$,
	and since $\ttN$ is a balanced binary tree, we can check for containment in time 
	$\calO(\log(\nec(\equiv^\pm_{d, A}))\cdot dn)$,
	so the total runtime of the algorithm is
	$\calO(\nec(\equiv^\pm_{d, A})\log\nec(\equiv^\pm_{d, A})\cdot dn^2)$.
\end{proof}

\subsection{Generalized Directed Domination Problems}\label{sec:alg:sigma-rho}
In this section we use the $d$-bi-neighborhood equivalence relation to give algorithms for problems that ask for a maximum- or minimum-size set that can be expressed as a $(\sigma^+, \sigma^-, \rho^+, \rho^-)$-set.
The algorithm is bottom-up dynamic programming along the given branch decomposition $(\decT, \decf)$ of the input digraph $G$, which we assume to be rooted in an arbitrary degree two node. For a node $t \in V(T)$, we let $V_t$ be the vertices of $G$ that are mapped to a leaf in the subtree of $T$ rooted at $t$.
Before we proceed with its description, 
we recall the formal definition here of 
$(\sigma^+, \sigma^-, \rho^+, \rho^-)$-sets.
\begin{definition}
	Let $\sigma^+, \sigma^-, \rho^+, \rho^- \subseteq \bN$, and let $\Sigma = (\sigma^+, \sigma^-)$ and $\Rho = (\rho^+, \rho^-)$.
	Let $G$ be a digraph and $S \subseteq V(G)$.
	We say that $S$ $(\sigma^+, \sigma^-, \rho^+, \rho^-)$-dominates $G$, or simply that $S$ $(\Sigma, \Rho)$-dominates $G$,
	if:
	\begin{align*}
		\forall v \in V(G) \colon \card{N^+(v) \cap S} \in \left\lbrace%
				\begin{array}{ll}
					\sigma^+, &\mbox{if } v \in S \\
					\rho^+, &\mbox{if } v \notin S
				\end{array}
				\right.
				\mbox{ and }~
				\card{N^-(v) \cap S} \in \left\lbrace%
				\begin{array}{ll}
					\sigma^-, &\mbox{if } v \in S \\
					\rho^-, &\mbox{if } v \notin S
				\end{array}
				\right.
	\end{align*}
\end{definition}

For better readability, it is often convenient to gather the sets $\sigma^+$ and $\sigma^-$ as one and the sets $\rho^+$ and $\rho^-$ as one.
We will mostly use the resulting $(\Sigma, \Rho)$-notation.
We now recall the definition of the $d$-value of a finite or co-finite set.
Informally speaking, this value tells us how far we have to count until in order to completely describe a finite or co-finite set.
\begin{definition}
	Let $d(\bN) = 0$.
	For a finite or co-finite set $\mu \subseteq \bN$, let 
	\[d(\mu) = 1 + \min\{\max_{x \in \bN} x \in \mu, \max_{x \in \bN} x \notin \mu\}.\]
	For finite or co-finite $\sigma^+, \sigma^-, \rho^+, \rho^- \subseteq \bN$, $\Sigma = (\sigma^+, \sigma^-)$ and $\Rho = (\rho^+, \rho^-)$ 
	\[d(\sigma^+, \sigma^-, \rho^+, \rho^-) = d(\Sigma, \Rho) = \max\{d(\sigma^+), d(\sigma^-), d(\rho^+), d(\rho^-)\}.\]
\end{definition}

As our algorithm progresses, it keeps track of partial solutions that may become a $(\Sigma, \Rho)$-set once the computation has finished.
This does not necessarily mean that at each node $t \in V(T)$, such a partial solution $X \subseteq V_t$ has to be a $(\Sigma, \Rho)$-dominating set of $G[V_t]$.
Instead, we additionally consider what is usually referred to as the ``expectation from the outside''~\cite{Bui-XuanTV13} in form of a subset $Y$ of $\bar{V_t}$ such that $X \cup Y$ is a $(\Sigma, \Rho)$-dominating set of $G[V_t]$. This is captured in the following definition.
\begin{definition}
	Let $\mu^+, \mu^- \subseteq \bN$ and let $\Mu = (\mu^+, \mu^-)$.
	Let $G$ be a digraph, and let $A \subseteq V(G)$ and $X \subseteq V(G)$.
	We say that $X$ $\Mu$-dominates $A$ if for all $v \in A$, 
	we have that $\card{N^+(v) \cap X} \in \mu^+$ and $\card{N^-(v) \cap X} \in \mu^-$.
	Let $\Sigma$ and $\Rho$ be as above.
	For $X \subseteq A$ and $Y \subseteq \overline{A}$,
	we say that $(X, Y)$ $(\Sigma, \Rho)$-dominates $A$,
	if $X \cup Y$ $\Sigma$-dominates $X$ and $X \cup Y$ $\Rho$-dominates $A \setminus X$.
\end{definition}

The next lemma shows that the previous definition behaves well with respect to $\equiv^\pm_{d, A}$.
\begin{lemma}\label{lem:dom:equiv}
	Let $\sigma^+, \sigma^-, \rho^+, \rho^- \subseteq \bN$ be finite or co-finite, 
	let $\Sigma = (\sigma^+, \sigma^-)$ and $\Rho = (\rho^+, \rho^-)$,
	and let $d = d(\Sigma, \Rho)$.
	Let $G$ be a digraph and let $A \subseteq V(G)$.
	Let $X \subseteq A$ and $Y, Y' \subseteq V(G) \setminus A$ such that $Y \equiv_{d, \overline{A}} Y'$.
	Then, $(X, Y)$ $(\Sigma, \Rho)$-dominates $A$ if and only if $(X, Y')$ $(\Sigma, \Rho)$-dominates $A$.
\end{lemma}
\begin{proof}
	Suppose $(X, Y)$ $(\Sigma, \Rho)$-dominates $A$.
	Let $v \in A \setminus X$.
	Since $Y \equiv_{d, \overline{A}}^- Y'$, we have that 
	\begin{align}
		\label{eq:1}
		\min\{d, \card{N^+(v) \cap Y}\} = \min\{d, \card{N^+(v) \cap Y'}\}.
	\end{align}
	If $\card{N^+(v) \cap Y} \le d$, then immediately by \eqref{eq:1} we have that
	$\card{N^+(v) \cap Y} = \card{N^+(v) \cap Y'}$.
	Therefore, $\card{N^+(v) \cap (X \cup Y)} = \card{N^+(v) \cap (X \cup Y')} \in \rho^+$.
	If $\card{N^+(v) \cap Y} > d$, then by the definition of the $d$-value
	we have that for all $n > d$, $n \in \rho^+$.
	By \eqref{eq:1}, this implies that $\card{N^+(v) \cap Y'} \in \rho^+$,
	and in particular that $\card{N^+(v) \cap (X \cup Y')} \in \rho^+$.
	Similarly we can show that $\card{N^-(v) \cap (X \cup Y')} \in \rho^-$,
	and so $X \cup Y'$ $\Rho$-dominates $A \setminus X$.
	Since $Y \equiv_{d, \overline{A}}^+ Y'$, we can use the same arguments to 
	show that for all $v \in X$, $\card{N^+(v) \cap (X \cup Y')} \in \sigma^+$ and $\card{N^-(v) \cap (X \cup Y')} \in \sigma^-$,
	and we conclude that $(X, Y')$ $(\Sigma, \Rho)$-dominates $A$.
	
	A symmetric argument yields the other direction, i.e., if $(X, Y')$ $(\Sigma, \Rho)$-dominates $A$
	then $(X, Y)$ $(\Sigma, \Rho)$-dominates $A$.
\end{proof}

We now turn to the definition of the table entries.
To describe an equivalence class $\calQ$ of $\equiv^\pm_{d, A}$
we use the $d$-bi-neighbohoods of its members.
Note that by \cref{obs:d-equiv:nbh},
the following notion of a description of an equivalence class
is well-defined.
\begin{definition}
	Let $G$ be a digraph, $A \subseteq V(G)$, and $d \in \bN$.
	For an equivalence class $\calQ$ of $\equiv^\pm_{d, A}$,
	its \emph{description}, denoted by $\desc(\calQ)$, 
	is the $d$-bi-neighborhood of all members of $\calQ$.
\end{definition}

As the table entries are indexed by equivalence classes of $\equiv^\pm_{d, A}$,
we use their descriptions as compact representations.
\begin{definition}
	Let $\sigma^+, \sigma^-, \rho^+, \rho^- \subseteq \bN$ be finite or co-finite, 
	let $\Sigma = (\sigma^+, \sigma^-)$, 
	$\Rho = (\rho^+, \rho^-)$, 
	and $d = d(\Sigma, \Rho)$.
	Let $\opt$ stand for $\min$ if we consider a minimization problem 
	and for $\max$ if we consider a maximization problem.
	Let $G$ be a digraph with branch decomposition $(\decT, \decf)$ and let $t \in V(\decT)$.
	For an equivalence class $\calQ_t$ of $\equiv^\pm_{d, V_t}$,
	and an equivalence class $\calQ_{\bar{t}}$ of $\equiv^\pm_{d, \bar{V_t}}$,
	we let:
	\begin{align*}
		\dptab_t[\desc(\calQ_t), \desc(\calQ_{\bar{t}})] = 
			\left\lbrace%
			\begin{array}{ll}
				\opt_{S \subseteq V_t} \card{S} \colon &S \in \calQ_t \mbox{ and for any } S_{\bar{t}} \in \calQ_{\bar{t}}\colon \\
						&(S, S_{\bar{t}}) \mbox{ $(\Sigma, \Rho)$-dominates } V_t  \\
					 \infty &\mbox{if } \opt = \min \mbox{ and no such $S$ exists} \\
					 -\infty &\mbox{if } \opt = \max \mbox{ and no such $S$ exists}
			\end{array}
			\right.
	\end{align*}
	We use the shorthand `$\dptab_t[\calQ_t, \calQ_{\bar{t}}]$' for `$\dptab_t[\desc(\calQ_t), \desc(\calQ_{\bar{t}})]$'.
\end{definition}

We first initialize the table entries for all $t \in V(\decT)$ as follows.
We use the algorithm of \cref{lem:enum:desc}
to enumerate all descriptions of equivalence classes $\calQ_t$ of $\equiv^\pm_{d, V_t}$
and $\calQ_{\bar{t}}$ of $\equiv^\pm_{d, \bar{V_t}}$
and we let:
\begin{align}
	\label{eq:sigma-rho:init}
	\dptab_t[\calQ_t, \calQ_{\bar{t}}] = \left\lbrace%
		\begin{array}{ll}%
			-\infty, &\mbox{ if } \opt = \max, \mbox{ and} \\
			\infty, &\mbox{ if } \opt = \min
		\end{array}
		\right.
\end{align}

\paragraph*{Leaves of $\decT$.} 
For a leaf $\ell \in V(\decT)$, let $v \in V(G)$ be such that $\decf(v) = \ell$.
Clearly, $\equiv_{d, \{v\}}$ has only two equivalence classes, 
namely the one containing $\emptyset$ and the one containing $\{v\}$.
For each equivalence class $\calQ$ of $\equiv_{d, V(G) \setminus \{v\}}$, 
let $R \in \calQ$ which we can assume is given to us by \cref{lem:enum:desc}.
\begin{itemize}
	\item If $\card{N^+(v) \cap R} \in \sigma^+$ and $\card{N^-(v) \cap R} \in \sigma^-$,
		then $\dptab_\ell[\{\{v\}\}, \calQ] = 1$.
	\item If $\card{N^+(v) \cap R} \in \rho^+$ and $\card{N^-(v) \cap R} \in \rho^-$, 
		then $\dptab_\ell[\{\emptyset\}, \calQ] = 0$.
\end{itemize}

Before we proceed with the description of the algorithm updating the table entries at internal nodes, we give one more auxiliary observation.
\begin{observation}\label{obs:eqc:triples}
	Let $d \in \bN$,
	and let $G$ be a digraph with a $3$-partition $(A, B, W)$ of $V(G)$.
	For each equivalence class $\calQ_a$ of $\equiv^{\pm}_{d, A}$ 
	and each equivalence class $\calQ_b$ of $\equiv^{\pm}_{d, B}$,
	the following holds.
	There is an equivalence class $\calQ$ of $\equiv^{\pm}_{d, A \cup B}$
	such that for all $R_a \in \calQ_a$ and $R_b \in \calQ_b$,
	$R_a \cup R_b \in \calQ$.
	Moreover, given a description of $\calQ_a$ and a description of $\calQ_b$,
	we can compute a description of $\calQ$ in time $\calO(\card{W})$.
\end{observation}
\begin{proof}
    The first statement follows immediately from the definitions.
    For the second statement, let $D = \desc(\calQ)$, $D_a = \desc(\calQ_a)$, and $D_b = \desc(\calQ_b)$. Then, for all $v \in W$, we have $D(v) = D_a(v) + D_b(v)$.
\end{proof}

\paragraph*{Internal nodes of $\decT$.}
Let $t \in V(\decT)$ be an internal node with children $a$ and $b$.
\begin{enumerate}[label={\arabic*.},ref={\arabic*}]
	\item Consider each triple $\calQ_a, \calQ_b, \calQ_{\bar{t}}$ of equivalence classes of
		$\equiv^\pm_{d, V_a}$, $\equiv^\pm_{d, V_b}$, and $\equiv^\pm_{d, \bar{V_t}}$, respectively.
	\item Let $R_a \in \calQ_a$, $R_b \in \calQ_b$, and $R_{\bar{t}} \in \calQ_{\bar{t}}$. 
		Determine:
		\begin{itemize}
			\item $\calQ_{\bar{a}}$, the equivalence class of $\equiv^\pm_{d, \bar{V_a}}$ containing $R_b \cup R_{\bar{t}}$.
			\item $\calQ_{\bar{b}}$, the equivalence class of $\equiv^\pm_{d, \bar{V_b}}$ containing $R_a \cup R_{\bar{t}}$.
			\item $\calQ_{t}$, the equivalence class of $\equiv^\pm_{d, V_t}$ containing $R_a \cup R_b$.
		\end{itemize}
	\item Update $\dptab_t[\calQ_t, \calQ_{\bar{t}}] = \opt\{\dptab_t[\calQ_t, \calQ_{\bar{t}}], 
			\dptab_a[\calQ_a, \calQ_{\bar{a}}] + \dptab_b[\calQ_b, \calQ_{\bar{b}}]\}$.
\end{enumerate}

The next two lemmas establish the correctness of the above algorithm.
\begin{lemma}\label{lem:sigma-rho:3parts}
	Let $\Sigma, \Rho$ be as above.
	Let $G$ be a digraph and let $(A, B, W)$ be a $3$-partition of $V(G)$.
	Let $S_a \subseteq A$, $S_b \subseteq B$, and $S_w \subseteq W$.
	Then $(S_a, S_b \cup S_w)$ $(\Sigma, \Rho)$-dominates $A$ and
	$(S_b, S_a \cup S_w)$ $(\Sigma, \Rho)$-dominates $B$ if and only if 
	$(S_a \cup S_b, S_w)$ $(\Sigma, \Rho)$-dominates $A \cup B$.
\end{lemma}
\begin{proof}
	Suppose $(S_a, S_b \cup S_w)$ $(\Sigma, \Rho)$-dominates $A$ and
	$(S_b, S_a \cup S_w)$ $(\Sigma, \Rho)$-dominates $B$;
	let $S = S_a \cup S_b \cup S_w$.
	Then, $S$ $\Sigma$-dominates $S_a$ and $S_b$, and therefore $S$ $\Sigma$-dominates $S_a \cup S_b$.
	Moreover, $S$ $\Rho$-dominates $A \setminus S_a$ and $B \setminus S_b$,
	so $S$ $\Rho$-dominates $(A \cup B) \setminus (S_a \cup S_b)$
	which yields that $(S_a \cup S_b, S_w)$ $(\Sigma, \Rho)$-dominates $A \cup B$.
	The other direction follows similarly.
\end{proof}

\begin{lemma}\label{lem:sigma-rho:correctness}
	For each node $t \in V(\decT)$, 
	the table entries in $\dptab_t$ are computed correctly.
\end{lemma}
\begin{proof}
	We prove the lemma by induction on the height of $t$.
	In the base case, $t$ is a leaf. 
	Correctness in this case is immediate.
	Suppose $t$ is an internal node with children $a$ and $b$.
	
	For the first direction, assume that $\dptab_t[\calQ_t, \calQ_{\bar{t}}] = k$
	for an equivalence class $\calQ_t$ of $\equiv^\pm_{d, V_t}$ 
	and an equivalence class $\calQ_{\bar{t}}$ of $\equiv^\pm_{d, \bar{V_t}}$.
	We show that in this case,
	there is a set $S \in \calQ_t$ of size $k$ 
	such that for all $S_{\bar{t}} \in \calQ_t$, 
	$(S, S_{\bar{t}})$
	$(\Sigma, \Rho)$-dominates $V_t$.
	By the update of the internal nodes
	and by the induction hypothesis,
	there are equivalence classes $\calQ_a$ of $\equiv^\pm_{d, V_a}$, $\calQ_{\bar{a}}$ of $\equiv^\pm_{d, \bar{V_a}}$,
	$\calQ_b$ of $\equiv^\pm_{d, V_b}$, and $\calQ_{\bar{b}}$ of $\equiv^\pm_{d, \bar{V_b}}$ such that:
	There exist $S_a \in \calQ_a$, $S_b \in \calQ_b$ with $\card{S_a} + \card{S_b} = k$
	such that for all $S_{\bar{a}} \in \calQ_{\bar{a}}$ and $S_{\bar{b}} \in \calQ_{\bar{b}}$,
	$(S_a, S_{\bar{a}})$ $(\Sigma, \Rho)$-dominates $V_a$,
	and $(S_b, S_{\bar{b}})$ $(\Sigma, \Rho)$-dominates $V_b$.
	Additionally, 
	$S_a \cup R_{\bar{t}} \in \calQ_{\bar{b}}$ and
	$S_b \cup R_{\bar{t}} \in \calQ_{\bar{a}}$,
	where $R_{\bar{t}} \in \calQ_{\bar{t}}$.
	Using \cref{lem:dom:equiv},
	we conclude that $(S_a, S_b \cup R_{\bar{t}})$ $(\Sigma, \Rho)$-dominates $V_a$
	and that $(S_b, S_a \cup R_{\bar{t}})$ $(\Sigma, \Rho)$-dominates $V_b$.
	\cref{lem:sigma-rho:3parts}
	yields that $(S_a \cup S_b, R_{\bar{t}})$ $(\Sigma, \Rho)$-dominates $V_a \cup V_b = V_t$.
	
	For the other direction, suppose that $\opt = \min$ and note that the case of $\opt = \max$ is analogous.
	We have to show for every pair of an equivalence class $\calQ_t$ of $\equiv^\pm_{d, V_t}$ 
	and an equivalence class $\calQ_{\bar{t}}$ of $\equiv^\pm_{d, \bar{V_t}}$, 
	and $R_{\bar{t}} \in \calQ_{\bar{t}}$
	that if there exists some $S_t \in \calQ_t$ of size at most $k$
	such that $(S_t, R_{\bar{t}})$ $(\Sigma, \Rho)$-dominates $V_t$,
	then $\dptab_t[\calQ_t, \calQ_{\bar{t}}] \le \card{S_t}$.
	Let $S_a \defeq S_t \cap V_a$ and $S_b \defeq S_t \cap V_b$.
	At some point, the algorithm considered the equivalence classes $\calQ_a$ of $\equiv^\pm_{d, V_a}$
	and $\calQ_b$ of $\equiv^\pm_{d, V_b}$
	such that $S_a \in \calQ_a$ and $S_b \in \calQ_b$.
	Since $(S_a \cup S_b, R_{\bar{t}})$ $(\Sigma, \Rho)$-dominates $V_a \cup V_b$,
	it follows from \cref{lem:sigma-rho:3parts} that 
	$(S_a, S_b \cup R_{\bar{t}})$ $(\Sigma, \Rho)$-dominates $V_a$.
	Note that by the above algorithm, $S_b \cup R_{\bar{t}} \in \calQ_{\bar{a}}$,
	so by the induction hypothesis we have that 
	$\dptab_a[\calQ_a, \calQ_{\bar{a}}] \le \card{S_a}$.
	Similarly we can deduce that $\dptab_b[\calQ_b, \calQ_{\bar{b}}] \le \card{S_b}$.
	Clearly, $S_a \cup S_b = S_t \in \calQ_t$,
	so the algorithm above guarantees that $\dptab_t[\calQ_t, \calQ_{\bar{t}}] \le \card{S_t}$.
\end{proof}

The algorithm described above results in the following theorem which is the first main result of this section.
\begin{theorem}\label{thm:sigma-rho}
	Let $\sigma^+, \sigma^-, \rho^+, \rho^- \subseteq \bN$ be finite or co-finite,
	$\Sigma = (\sigma^+, \sigma^-)$, $\Rho = (\rho^+, \rho^-)$, and $d = d(\Sigma, \Rho)$.
	There is an algorithm that given a digraph $G$ on $n$ vertices
	together with one of its branch decompositions $(\decT, \decf)$,
	computes and optimum-size $(\Sigma, \Rho)$-dominating set in time 
	$\calO(\nec_d(\decT, \decf)^3 \cdot n^3\log n)$.
	For $n \le \nec_d(\decT, \decf)$,
	the algorithm runs in time $\calO(\nec_d(\decT, \decf)^3 \cdot n^2)$.
\end{theorem}
\begin{proof}
	We assume $\decT$ to be rooted in an arbitrary node $r \in V(\decT)$ of degree two,
	and do bottom-up dynamic programming along $\decT$.
	We first initialize the table entries at all nodes as described in Equation~\eqref{eq:sigma-rho:init}.
	At each node $t \in V(\decT)$, 
	we perform the update of all table entries as described above in the corresponding paragraph,
	depending on whether $t$ is a leaf or an internal node.
	We find the solution to the instance at hand at the table entry $\dptab_r[2^{V(G)}, \{\emptyset\}]$.
	
	Correctness of the algorithm follows from \cref{lem:sigma-rho:correctness}; 
	we now analyze its runtime. We may assume that $\card{V(\decT)} = \calO(n)$.
	By \cref{lem:enum:desc},
	we can compute all descriptions of the equivalence classes of all equivalence relations associated with the nodes of 
	$\decT$ in time at most 
	$\calO(\nec_d(\decT, \decf)\log \nec_d(\decT, \decf)\cdot n^3) = \calO(\nec_d(\decT, \decf)\cdot d\cdot n^3\log n)$.
	(Note that $\log\nec_d(\decT, \decf) \le d\log n$.)

	Initialization of all table entries takes time at most $\calO(\nec_d(\decT, \decf)^2\cdot n)$.
	Updating the entries at leaf nodes takes time $\calO(d)$ per node, so at most $\calO(d \cdot n)$ in total, for all leaf nodes.
	At each internal node, we consider triples of equivalence classes, 
	of which there are at most $\nec_d(\decT, \decf)^3$.
	Once such a triple is fixed, the remaining computations can be done in time $\calO(n)$ by \cref{obs:eqc:triples},
	with an overhead of $\calO(d \log n)$ for querying the table entries,
	using a binary search tree.
	Therefore, for each triple of equivalence classes, the update takes at most $\calO(n + d \log n) = \calO(d n)$ time;
	which implies that updating the table entries at all internal nodes 
	takes time at most $\calO(\nec_d(\decT, \decf)^3\cdot d\cdot n^2)$.
	Since $d$ is a constant, we can bound the runtime of the algorithm by
	\begin{align*}
		\max\{\calO(\nec_d(\decT, \decf) \cdot n^3\log n), \calO(\nec_d(\decT, \decf)^3 \cdot n^2)\}.
	\end{align*}
\end{proof}

\subsection*{Runtime in terms of bi-mim-width}
We now show how to express the runtime of the algorithm from \cref{thm:sigma-rho} as an \XP-runtime parameterized by the bi-mim-width of $(\decT, \decf)$,
in analogy with the case of undirected graphs~\cite{BelmonteV2013}. 
The crucial observation is the following.
\begin{observation}\label{obs:nec:mim}
	 For $d \in \bN$, a digraph $G$, and $A \subseteq V(G)$:
	 $\nec({\equiv^\pm_{d, A}}) \le n^{d \cdot \bimimval_G(A)}$.
\end{observation}
\begin{proof}
	By the same arguments given in~\cite{BelmonteV2013} for the undirected case,
	we can show that 
	\begin{align*}
		\nec(\equiv_{d, A}^+) \le n^{d \cdot \mimval_G^+(A)} 
		\mbox{ and that }
		\nec(\equiv_{d, A}^-) \le n^{d \cdot \mimval_G^-(A)}.
	\end{align*}
	Therefore, 
	\begin{align*}
		\nec(\equiv^{\pm}_{d, A}) = &\nec(\equiv_{d, A}^+) \cdot \nec(\equiv_{d, A}^-)
			\le n^{d \cdot \mimval_G^+(A)} \cdot n^{d \cdot \mimval_G^-(A)}
			= n^{d \cdot \bimimval_G(A)}.
	\end{align*}
	%
\end{proof}

\begin{corollary}\label{cor:sigma-rho:bimim}
	Let $\sigma^+, \sigma^-, \rho^+, \rho^- \subseteq \bN$ be finite or co-finite,
	$\Sigma = (\sigma^+, \sigma^-)$, $\Rho = (\rho^+, \rho^-)$, and $d = d(\Sigma, \Rho)$.
	Let $G$ be a digraph on $n$ vertices with branch decomposition $(\decT, \decf)$
	of bi-mim-width $w \ge 1$.
	There is an algorithm that given any such $G$ and $(\decT, \decf)$
	computes an optimum-size $(\Sigma, \Rho)$-dominating set in time 
	$\calO(n^{3dw + 2})$.
\end{corollary}

\subsection{Directed Vertex Partitioning Problems}\label{sec:alg:lcvp}
We now show that the locally checkable vertex partitioning problems can be solved in \XP time paramterized by the bi-mim-width of a given branch decomposition.
In analogy with~\cite{Bui-XuanTV13}, we lift the $d$-bi-neighborhood equivalence to $q$-tuples over vertex sets, which allows for devising the desired dynamic programming algorithm.
We omit several technical details as they are very similar to the ones in the previous section.
We begin by recalling the definition of a bi-neighborhood constraint matrix.
\begin{definition}
	A \emph{bi-neighborhood-constraint} matrix is a $(q \times q)$-matrix $D_q$ 
	over pairs of finite or co-finite sets of natural numbers.
	Let $G$ be a digraph, and $\calX = (X_1, \ldots, X_q)$ be a $q$-partition of $V(G)$.
	We say that $\calX$ is a \emph{$D$-partition} if for all $i, j \in \{1, \ldots, q\}$
	with $D_q[i, j] = (\mu^+_{i, j}, \mu^-_{i, j})$, we have that for all $v \in X_i$,
	$\card{N^+(v) \cap X_j} \in \mu^+_{i, j}$ and $\card{N^-(v) \cap X_j} \in \mu^-_{i, j}$.
	The \emph{$d$-value} of $D_q$ is $d(D_q) = \max_{i, j}\{d(\mu^+_{i, j}), d(\mu^-_{i, j})\}$.
\end{definition}

\begin{definition}
	Let $G$ be a digraph and $A \subseteq V(G)$. 
	Two $q$-tuples of subsets of $A$, 
	$\calX = (X_1, \ldots, X_q)$ and $\calY = (Y_1, \ldots, Y_q)$, 
	are \emph{$d$-bi-neighborhood equivalent} w.r.t.~$A$, if
	\begin{align*}
		\forall i \forall v \in \bar{A} \colon 
			&\min\{d, \card{N^+(v) \cap X_i}\} = \min\{d, \card{N^+(v) \cap Y_i}\} \mbox{ and} \\
			&\min\{d, \card{N^-(v) \cap X_i}\} = \min\{d, \card{N^-(v) \cap Y_i}\}.
	\end{align*}
	In this case we write $\calX \equiv^\pm_{q, d, A} \calY$.
\end{definition}

\begin{observation}\label{obs:nec:q}
	Let $G$ be a digraph, $A \subseteq V(G)$, and let 
	$\calX = (X_1, \ldots, X_q) \in (2^A)^q$ and $\calY = (Y_1, \ldots, Y_q) \in (2^A)^q$.
	Then, $\calX \equiv^\pm_{q, d, A} \calY$ if and only if for all $i \in \{1, \ldots, q\}$,
	$X_i \equiv^\pm_{d, A} Y_i$.
	Therefore, $\nec(\equiv^\pm_{q, d, A}) = \nec(\equiv^\pm_{d, A})^q$.
\end{observation}

For a $q$-tuple $\calX = (X_1, \ldots, X_q)$ of subsets of $A$,
the \emph{$(q, d)$-bi-neighborhood} of $\calX$ w.r.t.~$A$ is
$\nbhv^{\pm}_{q, d, A}(\calX) = (\nbhv^{\pm}_{d, A}(X_1), \ldots, \nbhv^{\pm}_{d, A}(X_q))$,
and we denote the set of all $(q, d)$-bi-neighborhoods w.r.t.~$A$ by $\calnbhv^\pm_{q, d, A}$.
By the previous observation, for any $A \subseteq V(G)$, 
there is a natural bijection between the $(q, d)$-bi-neighborhoods w.r.t.~$A$
and the equivalence classes of $\equiv^\pm_{q, d, A}$.
Furthermore, we can enumerate $\calnbhv^{\pm}_{q, d, A}$ by invoking the algorithm of \cref{lem:enum:desc}
$q$ times and then generating all $q$-tuples of $\calnbhv^{\pm}_{d, A}$.
\begin{corollary}\label{cor:enum:q:desc}
	Let $G$ be a digraph on $n$ vertices, $A \subseteq V(G)$, and $q, d \in \bN$ with $q \ge 2$.
	There is an algorithm that enumerates all members of $\calnbhv^\pm_{q, d, A}$ in time 
	$\calO(\nec(\equiv^\pm_{d, A})^q\cdot qdn^2)$.
	Furthermore, for each $\calY \in \calnbhv^\pm_{q, d, A}$,
	the algorithm provides some $\calX \in (2^A)^q$ with $\nbhv^\pm_{q, d, A}(\calX) = \calY$.
\end{corollary}

\begin{definition}
	Let $D_q$ be a bi-neighborhood constraint matrix.
	Let $G$ be a digraph and $A \subseteq V(G)$.
	Let $\calX = (X_1, \ldots, X_q) \in (2^A)^q$ and $\calY = (Y_1, \ldots, Y_q) \in (2^{\bar{A}})^q$.
	We say that $(\calX, \calY)$ $D_q$-dominates $A$ if for all $i, j \in \{1, \ldots, q\}$,
	$X_i \cup Y_j$ $D_q[i, j]$-dominates $A$.
\end{definition}

For an equivalence class $\calQ$ of $\equiv^{\pm}_{q, d, A}$,
its \emph{description}, $\desc(\calQ)$, is the $(q, d)$-bi-neighborhood of all its members.
Again we index the table entries with descriptions of equivalence classes.
For a clearer presentation we will also here skip explicit mentions of the $\desc$-operator.
\begin{definition}
	Let $D_q$ be a bi-neighborhood constraint matrix
	with $d(D_q) = d$,
	and let $G$ be a digraph with branch decomposition $(\decT, \decf)$, and $t \in V(\decT)$.
	Let $\calQ_t$ be an equivalence class of $\equiv^\pm_{q, d, V_t}$ 
	and $\calQ_{\bar{t}}$ be an equivalence class of $\equiv^\pm_{q, d, \bar{V_t}}$.
	Then,
	\begin{align*}
		\dptab_t[\calQ_t, \calQ_{\bar{t}}] \defeq \left\lbrace
			\begin{array}{ll}
				\texttt{True}, &\mbox{if $\exists$ $q$-partition $\calS$ of $V_t$ such that:}\\
								&\mbox{$\calS \in \calQ_t$ and for all $\calR_{\bar{t}} \in \calQ_{\bar{t}}\colon 
									(\calS, \calR_{\bar{t}})$ $D_q$-dominates $V_t$;} \\
				\texttt{False}, &\mbox{otherwise.}
			\end{array}
			\right.
	\end{align*}
\end{definition}

By the definition of the table entries we have that $G$ has a $D_q$-partition 
if and only if some entry in $\dptab_r$ is true, where $r$ is the root of the
given branch decomposition of $G$.
We now describe the algorithm.
Initially, we set all table entries at all nodes to \texttt{False}.

\paragraph*{Leaves of $\decT$.}
If $\ell \in V(\decT)$ is a leaf of $\decT$, then let $v \in V(G)$ be such that $\decf(v) = \ell$.
We have to consider the following $q$ $q$-partitions of $\{v\}$ 
(recall that parts of a partition may be empty):
For $i \in \{1, \ldots, q\}$, we have to consider the partition $\calX_i = (X_1, \ldots, X_q)$ 
where $X_i = \{v\}$ and for $j \neq i$, $X_j = \emptyset$.
While these partitions are equal up to renaming, 
they might differ with respect to $D_q$.
We have that $\nbhv^{\pm}_{q, d, \{v\}}(\calX_i)$ 
is the all-zeroes vector in coordinates $j \neq i$,
and $\nbhv^{\pm}_{d, A}(\{v\})$ in coordinate $i$.
We denote the corresponding equivalence class of $\equiv^\pm_{q, d, \{v\}}$ by $\calQ_i$.
For each equivalence class $\calQ_{\bar{\ell}}$ of $\equiv^\pm_{q, d, V(G) \setminus \{v\}}$,
let $\calY = (Y_1, \ldots, Y_\ell)$ be one of its elements.
We then perform the following updates:
For all $j$, let $(\mu^+_{i, j}, \mu^-_{i, j}) = D_q[i, j]$; then
\begin{align*}
	\dptab_\ell[\calQ_i, \calQ_{\bar{\ell}}] = \mathtt{True} \mbox{ if }
		\forall j\colon \card{N^+(v) \cap Y_j} \in \mu^+_{i, j} \mbox{ and } \card{N^-(v) \cap Y_j} \in \mu_{i, j}^-.
\end{align*}

\newcommand\qunion{\bigcup_q}

In the following, for two $q$-tuples $(X_1, \ldots, X_q)$ and $(Y_1, \ldots, Y_q)$,
we denote their coordinate-wise union as 
$(X_1, \ldots, X_q) \qunion (Y_1, \ldots, Y_q) = (X_1 \cup Y_1, \ldots, X_q \cup Y_q)$.

\paragraph*{Internal nodes of $\decT$.}
Let $t \in V(\decT)$ be an internal node with children $a$ and $b$.
\begin{enumerate}[label={\arabic*.},ref={\arabic*}]
	\item Consider each triple $\calQ_a$, $\calQ_b$, $\calQ_{\bar{t}}$ of equivalence classes of 
		$\equiv^\pm_{q, d, V_a}$, $\equiv^\pm_{q, d, V_b}$, and $\equiv^\pm_{q, d, \bar{V_t}}$, respectively.
	\item Let $R_a \in \calQ_a$, $R_b \in \calQ_b$, and $R_{\bar{t}} \in \calQ_{\bar{t}}$. Determine:
	\begin{itemize}
		\item $\calQ_{\bar{a}}$, the equivalence class of $\equiv^{\pm}_{q, d, \bar{V_a}}$ containing $R_b \qunion R_{\bar{t}}$.
		\item $\calQ_{\bar{b}}$, the equivalence class of $\equiv^{\pm}_{q, d, \bar{V_b}}$ containing $R_a \qunion R_{\bar{t}}$.
		\item $\calQ_{t}$, the equivalence class of $\equiv^{\pm}_{q, d, V_t}$ containing $R_a \qunion R_b$.
	\end{itemize}
	\item If $\dptab_t[\calQ_t, \calQ_{\bar{t}}] = \mathtt{False}$,
		then update $\dptab_t[\calQ_t, \calQ_{\bar{t}}] = \dptab_a[\calQ_a, \calQ_{\bar{a}}] \wedge \dptab_b[\calQ_b, \calQ_{\bar{b}}]$.
\end{enumerate} 

Applying the arguments given in the proof of \cref{lem:sigma-rho:correctness} to each part of the corresponding partitions
yields the correctness of the resulting algorithm.
Also the running time can be analyzed in a similar way, 
using \cref{cor:enum:q:desc} and \cref{obs:nec:q}
to bound the complexity of enumerating all equivalence classes of 
the equivalence relations $\equiv^{\pm}_{q, d, \cdot}$.
We have the following theorem.
\begin{theorem}
	Let $D_q$ be a bi-neighborhood constraint matrix
	with $d = d(D_q)$.
	There is an algorithm that given a digraph $G$ on $n$ vertices
	together with one of its branch decompositions $(\decT, \decf)$,
	determines whether $G$ has a $D_q$-partition
	in time $\calO(\nec_d(\decT, \decf)^{3q} \cdot q\cdot n^3\log n)$.
	For $n \le \nec_d(\decT, \decf)$,
	the algorithm runs in time $\calO(\nec_d(\decT, \decf)^{3q} \cdot q\cdot n^2)$.
\end{theorem}

Combining the previous theorem with \cref{obs:nec:mim} gives the following
algorithms parameterized by the bi-mim-width of a given branch decomposition.
\begin{corollary}\label{cor:lcvp:bimim}
	Let $D_q$ be a bi-neighborhood constraint matrix
	with $d = d(D_q)$.
	Let $G$ be a digraph on $n$ vertices with branch decomposition $(\decT, \decf)$
	of bi-mim-width $w$.
	There is an algorithm that given any such $G$ and $(\decT, \decf)$
	decides whether $G$ has a $D_q$-partition
	in time $\calO(q\cdot n^{3qdw + 2})$.
\end{corollary}

\subsection{Distance-$r$ Variants}\label{sec:alg:distance-r}
We now turn to distance variants of all problems considered in this section so far. For instance, for $r \in \bN$, the \textsc{Distance-$r$ Dominating Set} problem asks for a minimum size set $S$ of vertices of a digraph $G$, such that each vertex in $V(G) \setminus S$ is at distance at most $r$ from a vertex in $S$.
Note that for $r = 1$, we recover the \textsc{Dominating Set} problem.
We can generalize all LCVS and LVCP problems to their distance-versions.

\begin{definition}
    Let $G$ be a digraph. 
    For $r \in \bN$, the \emph{$r$-out-ball} of a vertex $v \in V(G)$ is the set of vertices $B^+_r = \{w \in V(G) \mid \dist_G(v, w) \le r\}$,
    and the \emph{$r$-in-ball} of a vertex $v$ is the set of vertices $B^-_r(v) = \{u \in V(G) \mid \dist_G(u, v) \le r\}$.    
\end{definition}

In distance-$r$ versions of LCVS and LCVP problems, restrictions are posed on $B^+_r(v)$ instead of $N^+(v)$ and on $B^-_r(v)$ instead of $N^-(v)$.
\begin{definition}
    Let $r \in \bN$;
    let $\sigma^+, \sigma^-, \rho^+, \rho^-$ be finite or co-finite subsets of $\bN$, let $\Sigma = (\sigma^+, \sigma^-)$ and $\Rho = (\rho^+, \rho^-)$.
    Let $G$ be a digraph and $S \subseteq V(G)$. 
    We say that $S$ \emph{distance-$r$ $(\Sigma, \Rho)$-dominates $G$}, if:
   	\begin{align*}
		\forall v \in V(G) \colon \card{B_r^+(v) \cap S} \in \left\lbrace%
				\begin{array}{ll}
					\sigma^+, &\mbox{if } v \in S \\
					\rho^+, &\mbox{if } v \notin S
				\end{array}
				\right.
				\mbox{ and }~
				\card{B_r^-(v) \cap S} \in \left\lbrace%
				\begin{array}{ll}
					\sigma^-, &\mbox{if } v \in S \\
					\rho^-, &\mbox{if } v \notin S
				\end{array}
				\right.
	\end{align*}
\end{definition}

It is not difficult to see that a set $S \subseteq V(G)$ is a distance-$r$ $(\Sigma, \Rho)$-dominating set in $G$ if and only if $S$ is a $(\Sigma, \Rho)$-dominating set in $G^r$, the $r$-th power of $G$. Therefore, to solve \textsc{Distance-$r$ $(\Sigma, \Rho)$-Set} on $G$, we can simply compute $G^r$ and solve \textsc{$(\Sigma, \Rho)$-Set} on $G^r$.
By \cref{lem:powerbimim,cor:sigma-rho:bimim}, we have the following consequence.
\begin{corollary}\label{cor:dist:sigma-rho:bimim}
    Let $r \in \bN$;
	let $\sigma^+, \sigma^-, \rho^+, \rho^- \subseteq \bN$ be finite or co-finite,
	$\Sigma = (\sigma^+, \sigma^-)$, $\Rho = (\rho^+, \rho^-)$, and $d = d(\Sigma, \Rho)$.
	Let $G$ be a digraph on $n$ vertices with branch decomposition $(\decT, \decf)$
	of bi-mim-width $w \ge 1$.
	There is an algorithm that given any such $G$ and $(\decT, \decf)$
	computes an optimum-size distance-$r$ $(\Sigma, \Rho)$-dominating set in time 
	$\calO(n^{3drw + 2})$.
\end{corollary}

\begin{definition}
    Let $D$ be a $(q \times q)$ bi-neighborhood-constraint matrix.
    Let $G$ be a digraph. 
    A $q$-partition $\calX = (X_1, \ldots, X_q)$ of $V(G)$ is a \emph{distance-$r$ $D$-partition} of $G$,
    if for all $i, j$, where $D[i, j] = (\mu^+_{i, j}, \mu^-_{i, j})$, we have that for all $v \in X_i$, $\card{B^+_r(v) \cap X_j} \in \mu^+_{i, j}$ and
    $\card{B^-_r(v) \cap X_j} \in \mu^-_{i, j}$.
\end{definition}

By similar reasoning as above and \cref{lem:powerbimim,cor:lcvp:bimim}, 
we have the following.
\begin{corollary}
	Let $r \in \bN$;
	let $D_q$ be a bi-neighborhood constraint matrix
	with $d = d(D_q)$.
	Let $G$ be a digraph on $n$ vertices with branch decomposition $(\decT, \decf)$
	of bi-mim-width $w$.
	There is an algorithm that given any such $G$ and $(\decT, \decf)$
	decides whether $G$ has a distance-$r$ $D_q$-partition
	in time $\calO(q\cdot n^{3qdrw + 2})$.
\end{corollary}

\section{Conclusion}
We introduced the digraph width measure bi-mim-width, and showed that directed LCVS and LCVP problems (and their distance-$r$ versions) can be solved in polynomial time if the input graph is given together with a branch decomposition of constant bi-mim-width.
We showed that several classes of intersection digraphs have constant bi-mim-width which adds a large number of polynomial-time algorithms for locally checkable problems related to domination and independence (given a representation) to the relatively sparse literature on the subject.
Intersection digraph classes such as interval digraphs seem too complex to give polynomial-time algorithms for optimization problems.
Our work points to reflexivity as a reasonable additional restriction to give successful algorithmic applications of intersection digraphs, while maintaining a high degree of generality. 
Note that for intersection digraphs, reflexivity gives more structure than just adding loops.

As our algorithms rely on a representation of the input digraphs being provided at the input, we are naturally interested in computing representations of intersection digraph classes of bounded bi-mim-width in polynomial time. So far, this is only known for (reflexive) interval digraphs.
We have shown bounds on the bi-mim-width of adjusted permutation and adjusted rooted directed path digraphs, but have not been able to show a bound on their reflexive counterparts, which we leave as an open question.

In Lemma~\ref{lem:powerbimim}, we proved that  the $r$-th power of a digraph of bi-mim-width $w$ has bi-mim-width at most $rw$. For undirected graphs, there is a bound that does not depend on $r$~\cite{JaffkeKST2019}. We leave an open question whether a bound that does not depend on $r$ exists.

\bibliographystyle{plain}
\bibliography{bimim}

\end{document}